\documentclass[12pt]{amsart}
\usepackage{amsmath, amssymb, amsthm}
\usepackage{comment, empheq}
\usepackage{mathdots, breqn}
\usepackage{graphicx, xcolor}
\usepackage{hyperref}
\usepackage{mathtools}
\usepackage{cleveref}
\numberwithin{equation}{section}
\newtheorem{theorem}{Theorem}
\newtheorem{lemma}{Lemma}[section]
\newtheorem{prop}{Proposition}[section]

\newtheorem{conj}{Conjecture}
\theoremstyle{remark}
\newtheorem{remark}{Remark}

\newcommand{\R}{\mathbb{R}}

\newcommand{\C}{\mathbb{C}}
\newcommand{\Q}{\mathbb{Q}}
\newcommand{\Z}{\mathbb{Z}}

\newcommand{\N}{\mathbb{N}}

\newcommand{\X}{\mathbb{X}}
\newcommand{\F}{\mathcal{F}}
\newcommand{\B}{\mathcal{B}}
\newcommand{\D}{\mathcal{D}}
\newcommand{\W}{\mathcal{W}}
\newcommand{\GL}{\mathrm{GL}}

\newcommand{\PGL}{\mathrm{PGL}}
\newcommand{\SL}{\mathrm{SL}}
\newcommand{\SU}{\mathrm{SU}}

\newcommand{\gen}{\mathrm{gen}}

\newcommand{\diag}{\mathrm{diag}}
\newcommand{\aut}{\mathrm{aut}}
\newcommand{\Ad}{\mathrm{Ad}}

\title{Applications of analytic newvectors for $\mathrm{GL}(n)$}

\author{Subhajit Jana}

\address{Max Planck Institute for Mathematics, Vivatgasse 7, 53111 Bonn, Germany.}
\email{subhajit@mpim-bonn.mpg.de}

\begin{document}

\begin{abstract}
We provide a few natural applications of the analytic newvectors, initiated in \cite{JN}, to some analytic questions in automorphic forms for $\mathrm{PGL}_n(\mathbb{Z})$ with $n\ge 2$, in the archimedean analytic conductor aspect. We prove an orthogonality result of the Fourier coefficients, a density estimate of the non-tempered forms, an equidistribution result of the Satake parameters with respect to the Sato--Tate measure, and a second moment estimate of the central $L$-values as strong as Lindel\"of on average. We also prove the random matrix prediction about the distribution of the low-lying zeros of automorphic $L$-function in the analytic conductor aspect. The new ideas of the proofs include the use of analytic newvectors to construct an approximate projector on the automorphic spectrum with bounded conductors and a soft local (both at finite and infinite places) analysis of the geometric side of the Kuznetsov trace formula.
\end{abstract}

\maketitle

\section{Introduction}

\subsection{Motivation}
Given a finite family $\F$ of automorphic representations $\pi$ of a reductive group $G$ several important questions can be formulated in terms of asymptotic behaviour of the average
\begin{equation}\label{average-problem}
|\F|^{-1}\sum_{\pi\in \F}A(\pi),
\end{equation}
as the size of the family $|\F|$ tends off to infinity. Here $A(\pi)$ can be chosen to be interesting arithmetic objects attached to $\pi$, e.g., Fourier coefficients $\lambda_\pi$, or powers of central $L$-values $|L(1/2,\pi)|^k$.
Often acquiring knowledge about the statistical behaviour of $A(\pi)$ by making such averages is easier to handle than understanding $A(\pi)$ individually, due to availability of strong analytic techniques e.g. (relative) trace formulae. The families usually are defined in terms of some intrinsic attributes of the automorphic representations $\pi$, such as, their levels (non-archimedean), weights, spectral parameters, Laplace eigenvalues (archimedean), or analytic conductors. A plethora of works have been done in these aspects on various higher rank and higher dimensional arithmetic locally symmetric spaces among which we refer to \cite{B1, B2, B3, BBR, BBM, BrM}.

Let $G$ be the group $\PGL_r(\R)$. Let $C(\pi)$ be the archimedean analytic conductor of an automorphic representation for $\PGL_r(\Z)$, defined by Iwaniec--Sarnak \cite{IS}, see \S\ref{l-function-condutor}. Consider the family $\F_X$ defined by 
$$\{\text{automorphic representation $\pi$ for }\PGL_r(\Z)\mid C(\pi)<X\}.$$ 
Obtaining an asymptotic formula or estimating size of \eqref{average-problem} when $\F=\F_X$ as $X\to \infty$
for various $A(\pi)$ is the motivating problem in this article. For instance, if $A(\pi)=|L(1/2,\pi)|^k$, an asymptotic formula for \eqref{average-problem} for large enough $k$ would yield a \emph{sub-convex} bound for $L(1/2,\pi)$ in terms of $C(\pi)$. This seems to be a potential approach to understand the infamous \emph{conductor-drop} issue. We refer to \cite{K} for a survey on various other averaging problems over $\F_X$.

As a side note, we record that $|\F_X|\asymp X^{r-1}$. Moreover, from Langlands classification of the continuous spectrum (described in the next subsection) and a few inductive applications of the Weyl law (with respect to the analytic conductors) imply that the size of the subfamily
$$\{\pi\text{ cuspidal }\mid \pi\in \F_X\}$$
is also $\asymp X^{r-1}$.

In this article our goal is to initiate understanding of a smoothened version of the averages as in \eqref{average-problem} by constructing an approximate projector on $\F_X$ using \emph{archimedean analytic newvectors} (see \S\ref{analytic-newvector}) and subsequently analysing the averages by means of the Kuznetsov trace formula.
An application of this kind in the question of weighted counting of automorphic forms for $\PGL_r(\Z)$, for $r\ge 2$, with bounded analytic conductors is provided in \cite[Theorem 9]{JN}. In this article we will give a few other averaging applications of this kind which (including their proofs) are mostly influenced by the recent works \cite{B1, B2, BBR, BBM, GK, GSWH, MT, ST}.

\subsection{Main Theorems}
Let $G:=\PGL_r(\R)$, $\Gamma:=\PGL_r(\Z)$, and $\X:=\Gamma\backslash G$. By $\hat{\X}$ we denote the isomorphism class of irreducible unitary standard automorphic representations of $G$ in $L^2(\X)$.
Here by \emph{standard}, we mean the automorphic representations which appear in the spectral decomposition of $L^2(\X)$. Similarly, by $\hat{\X}_\gen$ we denote the subclass of \emph{generic} representations in $\hat{\X}$ i.e., class of representations which have (unique) Whittaker models (see \S\ref{automorphic-forms}). 

We mention Langlands' descriptions for $\hat{\X}$ and $\hat{\X}_\gen$. We refer to \cite[Section 5]{CPS2}, \cite{MW} for details.
We take a partition $r=r_1+\dots+r_k$. Let $\pi_j$ be a unitary \emph{square-integrable} automorphic representation (see \cite{MW} for the classification of the square-integrable representations of $\GL(n)$) for $\GL_{r_j}(\Z)$ (if $r_j=1$ we take $\pi_j$ to be a unitary character). Consider the unitary induction $\Pi$ from the Levi $\GL(r_1)\times\dots\times\GL(r_k)$ to $G$ of the tensor product $\pi_1\otimes\dots\otimes\pi_k$. There exists a unique irreducible constituent of $\Pi$ which we denote by the isobaric sum $\pi_1\boxplus\dots\boxplus\pi_k$. Then Langlands classification says that every element in $\hat{\X}$ is isomorphic to such an isobaric sum. Further, we start with unitary \emph{cuspidal} automorphic representations $\pi_j$ for $\GL_{r_j}(\Z)$ and proceed with the same construction as above to obtain an isobaric sum. Then another theorem of Langlands asserts that every element in $\hat{\X}_\gen$ is isomorphic to such an isobaric sum.

Let $d\mu_\aut$ be the automorphic Plancherel measure on $\hat{\X}$ compatible with the $G$-invariant probability measure on $\X$ (see \cite[Chapter $11.6$]{G} for details). Let $C(\pi)$ be the analytic conductor of $\pi\in \hat{\X}_\gen$ (see \S\ref{l-function-condutor}).

\subsection{Orthogonality of Fourier coefficients}
The Fourier coefficients of automorphic forms on $\GL(r)$ for $r\ge 2$ behave similarly to the characters on $\GL(1)$. In particular, one expects that the Fourier coefficients satisfy an orthogonality relation when averaged over a sufficiently large family, similar to characters. In \cite[Conjecture 1.1]{Zh} an orthogonality conjecture of the Fourier coefficients is formulated, which loosely states that
\begin{equation}\label{conjecture-orthogonality}
\lim_{T\to \infty}\frac{\sum_{\varphi\text{ cuspidal, }\nu_\varphi\le T}{\lambda_\varphi(m)\overline{\lambda_\varphi(n)}}{L(1,\varphi,\Ad)}^{-1}}{\sum_{\varphi\text{ cuspidal, }\nu_\varphi\le T}{L(1,\varphi,\Ad)}^{-1}}=\delta_{m=n}.
\end{equation}
Here $\varphi$ are spherical (i.e. also unramified at infinity) cusp forms on $\X$, $\lambda_\varphi$ are the Fourier coefficients of $\varphi$, $\nu_\varphi$ are the Laplace eigenvalues of $\varphi$, and $m,n\in \N^{r-1}$.
For detailed notations see \S\ref{basic-notations}.

In Theorem \ref{orthogonality-full-spectrum} we prove a variant of the conjecture in \eqref{conjecture-orthogonality} in the analytic conductor aspect (i.e. the Laplace eigenvalue $\nu_\varphi$ replaced by the analytic conductor of $\varphi$) for general $r$.

\begin{theorem}\label{orthogonality-full-spectrum}
Let $m,n\in \N^{r-1}$ and $X>1$ be a large number tending off to infinity, such that 
$$\min(n_{r-1}^{r-1}\dots n_1m_1^{r-1}\dots m_{r-1},m_{r-1}^{r-1}\dots m_1 n_1^{r-1}\dots n_{r-1})\ll X^r$$
with sufficiently small implied constant. There exists $J_X:\hat{\X}_\gen\to \R_{\ge 0}$ with 
\begin{itemize}
    \item $J_X(\pi)\gg 1$ whenever $\pi$ is cuspidal with $C(\pi)<X$,
    \item $\int_{\hat{\X}_\gen}\frac{J_X(\pi)}{\ell(\pi)}d\mu_\aut(\pi)\asymp X^{r-1}$,
\end{itemize}
such that,
$$\int_{\hat{\X}_\gen}\overline{\lambda_\pi(m)}{\lambda_\pi(n)}\frac{J_X(\pi)}{\ell(\pi)}d\mu_\aut(\pi)=\delta_{m=n}\int_{\hat{\X}_\gen}\frac{J_X(\pi)}{\ell(\pi)}d\mu_\aut(\pi),$$
where $\ell(\pi)$ is a positive number, defined in \eqref{defn-ell-pi}, depending only on the non-archimedean data attached to $\pi$.
\end{theorem}

We remark that the number of cuspidal representations of $G$ with analytic conductor bounded by $X$ is of size $X^{r-1}$. 

In \cite{Bru} the conjecture \eqref{conjecture-orthogonality} is first proved for $r=2$. Recently, in \cite[Theorem, 5]{B1}, \cite[Theorem 5]{BBR}, \cite{GK} for $r=3$, and in \cite{GSWH} for $r=4$ the conjecture has been settled in the Laplace eigenvalue aspect. On the other hand, a variant of the conjecture \eqref{conjecture-orthogonality} without the harmonic weights $L(1,\pi,\Ad)$ has been proved in \cite[Theorem 1.5]{MT} for general $r$.

\begin{remark}\label{project-on-cusp}
In Theorem \ref{orthogonality-full-spectrum}, unlike \eqref{conjecture-orthogonality}, we have averaged over not just the cuspidal spectrum but also included the continuous spectrum. However, one can modify the proof of Theorem \ref{orthogonality-full-spectrum} to have an orthogonality result over a subset of the cuspidal spectrum only, by killing off the contributions from the continuous spectrum using a projector attached to a matrix coefficient of a supercuspidal representation $\sigma$ of $\PGL_r(\Q_p)$ for some fixed prime $p$. We give such an example in Theorem \ref{orthogonality-cuspidal-spectrum} which loosely describes a statement of the following spirit. 

\noindent
\textbf{Theorem \ref{orthogonality-cuspidal-spectrum}}. If $m,n\in \N^{r-1}$ coprime with $p$ fixed, then
$$\lim_{X\to\infty}X^{1-r}\sum_{\substack{{C(\pi_\infty)<X,\pi_p=\sigma}\\{\pi\,\,\mathrm{cuspidal}}}}\frac{\overline{\lambda_\pi(m)}\lambda_\pi(n)} {L^{(p)}(1,\pi,\Ad)}=c_\sigma\delta_{m=n},$$
where $L^{(p)}$ is the partial $L$-function excluding the $p$-adic Euler factor, $c_\sigma$ is a constant depending on $\sigma$, and $\pi_p$ and $\pi_\infty$ are the $p$-adic and infinity components of an automorphic representation $\pi$ of $\PGL_r(\Z[1/p])$, respectively.

\end{remark}

\begin{remark}
We only show that the cut-off function $J_X(\pi)$ is non-negative over the relevant spectrum and large on the cuspidal spectrum if the analytic conductor is bounded by $X$. Although we do not show that $J_X(\pi)$ is negligible if the analytic conductor is large, we expect that it is true nevertheless (see \S\ref{analytic-newvector}).
\end{remark}

\subsection{Vertical Sato--Tate}
For a finite prime $p$ and $\pi\in \hat{\X}_\gen$ we denote the complex unordered $r$-tuple 
$$\mu^p(\pi):=(\mu^p_1(\pi),\dots,\mu^p_r(\pi)),\quad \sum_i\mu_i^p(\pi)=0,\mod 2\pi i/\log p,$$ 
to be the Langlands parameters attached to $\pi$ at $p$. Langlands parameters are invariant under the action of the Weyl group $W$. Let $T$ and $T_0$ be the standard maximal tori in $\mathrm{SL}_r(\C)$ and $\SU(r)$, respectively. We identify $\mu^p(\pi)$ as an element of $T/W$ by $\mu^p(\pi)\mapsto p^{\mu^p(\pi)}:=\mathrm{diag}(p^{\mu^p_1(\pi)},\dots,p^{\mu^p_r(\pi)})$. As a tuple $p^{\mu^p(\pi)}$ are the \emph{Satake parameters} attached to $\pi$ at $p$.
Let $\mu_{\mathrm{st}}$ be the push-forward of the Haar measure on $\SU(r)$ on $T_0/W$, which is also called to be the \emph{Sato--Tate} measure attached to the group $\GL(r)$. 

The Ramanujan conjecture at a finite prime $p$ for $\GL(r)$ predicts that all the Satake parameters will be purely imaginary, in other words, $p^{\mu^p(\pi)}$ will be in $T_0/W$, for all $\pi\in \hat{\X}_\gen$. Although the Ramanujan conjecture itself is still open, one can verify its truth on an average as in Theorem \ref{sato-tate}. For details we refer to \cite{Zh, BBR}. Theorem \ref{sato-tate} can also be regarded as a weighted equidistribution result of the Satake parameters with respect to the Sato--Tate measure. 

\begin{theorem}\label{sato-tate}
Let $f$ be a continuous function on $T/W$. Let $p$ be a finite prime and $p^{\mu^p(\pi)}\in T/W$ be the Satake parameters attached to $\pi$ at $p$. Then
$$\frac{\int_{\hat{\X}_\gen}f(p^{\mu^p(\pi)})\frac{J_X(\pi)}{\ell(\pi)}d\mu_\aut(\pi)}{\int_{\hat{\X}_\gen}\frac{J_X(\pi)}{\ell(\pi)}d\mu_\aut(\pi)}\to\int_{T_0/W}f(z)d\mu_{\mathrm{st}}(z),$$
as $X\to \infty$.
\end{theorem}

A variant of Theorem \ref{sato-tate} in the Laplace eigenvalue aspect assuming the conjecture in \eqref{conjecture-orthogonality} is proved in \cite{Zh} for general $r$. Unconditionally, the same variant is proved in \cite[Theorem 10.2]{KL} for $r=2$ and in \cite[Theorem 3]{BBR} for $r=3$. In \cite[Theorem 1.4]{MT}, and in \cite{ST} the authors proved Sato--Tate equidistribution over a cuspidal spectrum without the harmonic weights $\ell(\pi)$ for $\PGL(r)$ and for general reductive group, respectively. In particular, both \cite{MT,ST} use the Arthur--Selberg trace formula while we use the Kuznetsov trace formula. We also mention that Theorem \ref{sato-tate} can be formulated over a cuspidal spectrum following a similar technique as in Theorem \ref{orthogonality-cuspidal-spectrum} (see Remark \ref{project-on-cusp}).

\subsection{Density estimates}
We denote $\theta^p(\pi):=\max_i |\Re(\mu^p_i(\pi))|$. We call $\pi$ to be \emph{tempered} if it satisfies the Ramanujan conjecture at $p$ i.e., $\theta^p(\pi)=0$. Otherwise, if $\theta^p(\pi)=\theta$, we call $\pi$ to be \emph{$\theta$--non-tempered} at $p$. It is known from \cite{JS} that $p^{\mu^p(\pi)}\in T_1/W$ for cuspidal $\pi$ in $\hat{\X}_\gen$, where $T_0\subset T_1\subset T$ is defined by the subset in $T$ containing $p^{\mu^p(\pi)}$ with $\theta^p(\pi)\le1/2$ (Luo--Rudnick--Sarnak improved this bound to $1/2-1/(1+r^2)$, also see \cite{MS}).

In \cite[p.465]{Sa}, Sarnak conjectured that for a nice enough finite family $\mathcal{F}$ of unitary irreducible automorphic representations for $\GL_r(\Z)$ the number of representations which are at least $\theta$--non-tempered at a fixed place, is essentially of size $|\mathcal{F}|^{1-\frac{2\theta}{r-1}}$. We refer to \cite{B3} for motivation and details. In Theorem \ref{density} we prove Sarnak's conjecture in \cite{Sa} in the analytic conductor aspect for a finite place $p$ for $\PGL_r(\Z)$.

\begin{theorem}\label{density}
Let $p$ be a fixed prime. Then
$$\frac{1}{X^{r-1}}|\{\pi\,\,\mathrm{cuspidal}, C(\pi)<X\mid \theta^p(\pi)> \eta\}|\ll_\epsilon X^{-2\eta+\epsilon},$$
as $X$ tends off to infinity.
\end{theorem}

Many variants of the density estimate as in Theorem \ref{density} of similar strength are available in the literature. We refer to \cite[Theorem 2]{B1} for $r=3$ in the spectral parameter aspect, \cite[Theorem 1, Theorem 2]{BBR} for $r=3$ in the Laplace eigenvalue aspect, \cite[Theorem 4, Theorem 5]{BBM} for $r=3$ in the level aspect, and more recently, \cite[Theorem 1]{B2} in the level aspect for general $r$. In \cite[Corollary 1.8]{MT} (also see \cite{FM} which discusses this for general reductive group) a density bound is obtained using the Arthur--Selberg trace formula for general $r$, however the bound is weaker than Sarnak's density hypothesis in \cite{Sa}.

\begin{remark}
We mention that \cite[Theorem 1]{B2} proves a stronger estimate in the non-archimedean aspect than the non-archimedean variant of Sarnak's density hypothesis \cite{Sa}. The analogous estimate in our archimedean setting would be $\ll_\epsilon X^{-4\eta+\epsilon}$. To obtain a stronger bound one needs to have an estimate of certain double unipotent orbital integral which arises in the geometric side of the Kuznetsov formula (as in the second term of the RHS of the equation in Proposition \ref{kuznetsov}) similar to what is achieved in the non-archimedean case in \cite[Theorem 3]{B2} (see discussion after Theorem 2 in \cite{B2}). It seems such estimates might be achievable by a delicate stationary phase analysis of the unipotent orbital integrals, which is not in our grasp currently for general $r$ and we hope to come back to this in future.
\end{remark}

\subsection{Statistics of low-lying zeros and symmetry type}
There have been strong evidences so far in support of Hilbert and P\'olya's suggestion that there might be a spectral interpretation of the distribution of the zeros of the Riemann Zeta function in terms of the eigenvalues of the random matrices, e.g. Montgomery's pair correlation of high zeros of Riemann zeta function. Katz and Sarnak \cite{KS} predicted that given a family of $L$-function for a reductive group one can associate a symmetry type (e.g. orthogonal, unitary, or symplectic) to that family, which is given by the associated random matrix ensemble which conjecturally determines the distribution of the low-lying zeros of the $L$-functions in the family. We refer to \cite[\S1.3-\S1.6]{ST} for a detailed overview and motivations.

We start by assuming the \emph{Langlands strong functoriality principle}, as in \cite[Hypothesis 10.1]{ST}. Note that the Langlands $L$-group of $G$ is $\SL_r(\C)$. By $\rho$ we will denote the $L$-homomorphism $\SL_r(\C)\to\GL_d(\C)$ arising from a $d$-dimensional irreducible representation of $\SL_r(\C)$.

\begin{conj}[Langlands Functoriality]\label{conj-langlands}
For every $\pi\in \hat{\X}_\gen$ and every $L$-homomorphism $\rho:\SL_r(\C)\to \GL_d(\C)$ there exists an irreducible automorphic representation $\rho_*\pi$ for $\GL_d(\Z)$, which is called to be the functorial lift of $\pi$ under $\rho$, such that
$$\Lambda(s,\rho_*\pi,\mathrm{Standard})=\Lambda(s,\pi,\rho),$$ 
and for every $p\le\infty$ the $p$-component $(\rho_*\pi)_p=\rho^p_*\pi_p$ where $\rho^p_*$ in the RHS denotes the induced functorial transfer for the local representations at $p$.
\end{conj}
To exclude the case $\rho$ being trivial we also assume that $d>1$. It is known that Conjecture \ref{conj-langlands} implies the \emph{generalized Ramanujan conjecture} for $G$.

We define the \emph{average conductor} $C_{\rho,X}$ of the family of automorphic representations $\pi\in \hat{\X}_\gen$ with analytic conductor $C(\pi)<X$ by the equation
\begin{equation*}
    \log C_{\rho,X}{\int_{\hat{\X}_\gen} \frac{J_X(\pi)}{\ell(\pi)}d\mu_\aut(\pi)}:={\int_{\hat{\X}_\gen}\log C(\rho_*\pi) \frac{J_X(\pi)}{\ell(\pi)}d\mu_\aut(\pi)}.
\end{equation*}

We write the zeros of $\Lambda(s,\pi,\rho)$ in the critical strip, i.e. $\Re(s)\in [0,1]$, as $1/2+i\gamma_{\rho_*\pi}$, i.e. under the GRH we have $\gamma_{\rho_*\pi}\in \R$. The \emph{low-lying zeros} of $\Lambda$ are the zeros with $\gamma_{\rho_*\pi}$ bounded by $(\log C_{\rho,X})^{-1}$.

Let $\psi$ be the Fourier transform of a smooth function on $\R$ supported on the interval $[-\delta,\delta]$ for some $\delta>0$. Then $\psi$ is a Paley--Wiener type (or Schwartz class) function. We define a \emph{weighted $1$-level density} statistic $D_{\rho,X}$ for low-lying zeros of $\Lambda(s,\rho_*\pi)$ with $C(\pi)<X$ by the equation
\begin{equation}
    D_{\rho,X}(\psi){\int_{\hat{\X}_\gen} \frac{J_X(\pi)}{\ell(\pi)}d\mu_\aut(\pi)}:={\int_{\hat{\X}_\gen}\sum_{\gamma_{\rho_*\pi}}\psi\left(\gamma_{\rho_*\pi}\frac{\log C_{\rho, X}}{2\pi}\right)\frac{J_X(\pi)}{\ell(\pi)}d\mu_\aut(\pi)}.
\end{equation}
Similarly, we denote the poles of $\Lambda(s,\pi,\rho)$ in the critical strip by $1/2+i\tau_{\rho_*\pi}$ and define 
\begin{equation}
    D^{\mathrm{pole}}_{\rho,X}(\psi){\int_{\hat{\X}_\gen} \frac{J_X(\pi)}{\ell(\pi)}d\mu_\aut(\pi)}:={\int_{\hat{\X}_\gen}\sum_{\tau_{\rho_*\pi}}\psi\left(\tau_{\rho_*\pi}\frac{\log C_{\rho, X}}{2\pi}\right)\frac{J_X(\pi)}{\ell(\pi)}d\mu_\aut(\pi)}.
\end{equation}

From the Katz--Sarnak \cite{KS} random matrix heuristics about the distribution of the zeros of $\Lambda$ one can predict that there is a limiting $1$-level distribution of the low-lying zeros. That is,
$$\lim_{X\to\infty} D_{\rho,X}(\psi)=\int_\R \psi(x)\mathfrak{W}(x)dx,$$
where $\mathfrak{W}$ is determined by the \emph{Frobenius--Schur indicator} $\mathfrak{s}(\rho)$ of the representation $\rho$. In particular, a family is (even) orthogonal, (unitary) symplectic, or unitary if the value of $\mathfrak{s}$ is $-1,1$, or $0$, respectively. Shin--Templier \cite[Theorem 1.5]{ST} showed that the above prediction is true for certain families of automorphic $L$-functions in the level and weight aspects for a general reductive group, assuming some hypotheses on the size of the average conductor and the number of poles of the $L$-function.
In Theorem \ref{low-lying-zero} we prove a weighted version of \cite[Theorem 1.5]{ST} for the group $G$ assuming similar hypotheses as in \cite{ST}, which we describe below.

We make a conjecture about the size of the average conductor as in \cite[Hypothesis 11.4]{ST} that $\log C_{\rho,X}\asymp_\rho \log X$.
\begin{conj}\label{conj-cond-size}
For all $L$-homomorphisms $\rho$ there exist nonnegative constants $\mathfrak{c}(\rho)$ and $\mathfrak{C}(\rho)$ such that,
$$X^{\mathfrak{c}(\rho)}\ll C_{\rho,X}\ll X^{\mathfrak{C}(\rho)},$$
for all large enough $X>1$.
\end{conj}
Conjecture \ref{conj-cond-size} should probably not be hard to establish upon an asymptotic expansion of $J_X(\pi)$ for various ranges of $C(\pi)$ (see \S\ref{analytic-newvector}) and using local Langlands correspondence for $G$ to describe the archimedean $L$-parameters of $\rho_*\pi$. Heuristically, at least when $\rho$ is the standard representation, we expect
$${\int_{\hat{\X}_\gen}\log C(\pi) \frac{J_X(\pi)}{\ell(\pi)}d\mu_\aut(\pi)} \approx \log X\int_{C(\pi)\asymp X} \frac{J_X(\pi)}{\ell(\pi)}d\mu_\aut(\pi)\asymp X^{r-1}\log X.$$
For general $d$-dimensional representation $\rho$, it is not difficult to prove a trivial bound $C(\rho_*\pi)\le C(\pi)^d$ (however this bound is far from being sharp when $\rho_*\pi$ has a conductor drop). Correspondingly, a crude bound of $\mathfrak{C}(\rho)$ may be obtained as discussed in the heuristic above.

We assume Langlands functoriality in Conjecture \ref{conj-langlands} and Conjecture \ref{conj-cond-size} about the size of the average conductor and state the following weighted version of the density estimate of the low-lying zeros of $\Lambda(s,\pi,\rho)$ over the family $\pi\in\hat{\X}_\gen$ with $C(\pi)<X$.

\begin{theorem}\label{low-lying-zero}
We assume Conjecture \ref{conj-langlands} and Conjecture \ref{conj-cond-size}. Let $\rho$ be a $d$-dimensional non-trivial irreducible representation of $^{L}{G}=\SL_r(\C)$ with the highest weight $\theta:=(\theta_1,\dots,\theta_r)$ which is a dominant element in $\Z^r/\Z(1,\dots,1)$. Recall $\mathfrak{s}(\rho)$ and $\mathfrak{C}(\rho)$ from the above discussion. Let $\psi$ be a Schwartz class function with its Fourier transform $\hat{\psi}$ supported on $[-\delta,\delta]$. Then
$$D_{\rho,X}(\psi)-D^{\mathrm{pole}}_{\rho,X}(\psi)=\hat{\psi}(0)-\frac{\mathfrak{s}(\rho)}{2}\psi(0)+O\left(\frac{1}{\log X}\right),$$
for all $\delta<\frac{1}{\mathfrak{C}(\rho)(\theta_1-\theta_r)}$.
\end{theorem}

\begin{remark}
It should be noted that that $D^{\mathrm{pole}}$ term in Theorem \ref{low-lying-zero} typically should be negligible. In \cite{ST} it is shown that the corresponding $D^{\mathrm{pole}}$ in their analysis is negligible upon assuming \cite[Hypothesis 11.2]{ST} which heuristically predicts that $\Lambda(s,\rho_*\pi)$ is entire for \emph{for most} $\pi$. We refer to Remark \ref{rmk-d-pole} for details.
\end{remark}

\begin{remark}
We use Ramanujan Conjecture in full strength (which is implied by Langlands functoriality principle in Conjecture \ref{conj-langlands}) in the proof. However, we only need a bound towards Ramanujan. Thus alternatively, one can assume functoriality for fewer $\rho$ and prove a bound towards Ramanujan and consequently, to prove a stronger version of Theorem \ref{low-lying-zero}.
\end{remark}

Apart form \cite{ST}, there are a few results on the distribution of the low-lying zeros for various families of $L$-functions e.g., for $r=2$ \cite{ILS,Y} in the weight and level aspect, for $r=3$ \cite{GK} in the Laplace eigenvalue aspect, and for general $r$ \cite{MT} in the dilated Plancherel ball aspect.

\subsection{A large sieve inequality}
In the next application we prove a large sieve inequality in the analytic conductor aspect for Fourier coefficients of automorphic forms on $\PGL(r)$. This result is in the spirit of celebrated large sieve inequalities in \cite{DI} in the spectral parameter aspect for $r=2$.

\begin{theorem}\label{large-sieve}
Let all the notations be as in Theorem \ref{orthogonality-full-spectrum}. Let $\alpha(n)_{n\in\N^{r-1}}$ be any sequence of complex numbers and $N:=N(n):=n_1\dots n_{r-1}$. Then 
$$\sum_{\substack{{C(\pi)<X}\\{\pi\,\,\mathrm{cuspidal}}}}{L(1,\pi,\Ad)^{-1}}\left|\sum_{N\ll X}\alpha(n)\lambda_\pi(n)\right|^2\ll X^{r-1}\sum_{N\ll X}\left|\alpha(n)\right|^2.$$
Here the implicit constant in the condition $N\ll X$ is assumed to be sufficiently small.
\end{theorem}

Theorem \ref{large-sieve} is of similar strength as in \cite[Theorem 4]{B2} in the level aspect. We also mention previous works on large Sieve inequalities in \cite[Theorem 3]{B1} in the spectral parameter aspect, \cite[Theorem 2, Theorem 3]{BBM} in level aspect for $r=3$, and \cite{DK, V} in the level aspect for general $r$.

\subsection{Second moment of the central L-values}
Finally, we give a corollary of Theorem \ref{large-sieve} to the best possible, i.e. Lindel\"of on average, second moment estimate of the central $L$-values.

\begin{theorem}\label{second-moment}
Let $L(s,\pi)$ be the standard $L$-function attached to $\pi$. Then
$$\sum_{\substack{{C(\pi)<X}\\{\pi\,\,\mathrm{cuspidal}}}}\frac{|L(1/2,\pi)|^2}{L(1,\pi,\Ad)}\ll_\epsilon X^{r-1+\epsilon},$$
as $X$ tends off to infinity.
\end{theorem}

A level aspect variant of Theorem \ref{second-moment} is recently proved in \cite[Corollary 5]{B2}. One can get rid of the harmonic weight $L(1,\pi,\Ad)$ in the above second moment estimate (similarly also in Theorem \ref{large-sieve}) by using an upper bound of $L(1,\pi,\Ad)$ as in \eqref{bound-l-value-at-1} from \cite{Li}.

\subsection{Smoothened average using analytic newvectors}\label{analytic-newvector}
We end this section with a few words about the archimedean analytic newvectors and their application in proving above theorems.
The classical or non-archimedean newvector theory was pioneered by Casselman \cite{C} and Jacquet--Piatetski-Shapiro--Shalika \cite{JPSS1}. They showed that for every generic irreducible representation $\pi$ of $\PGL_r(\Q_p)$ there exists a nonzero vector in $\pi$ which is invariant under the open compact subgroup $K_0(p^{c(\pi)})$, where $c(\pi)$ is the conductor exponent of $\pi$ and $K_0(p^N)$ is the subgroup of $\PGL_r(\Z_p)$ consisting matrices whose last row is congruent to $(0,\dots,0,*)\mod p^N$. Several works in the literature have used such invariance of newvectors when working on a family of automorphic forms with bounded conductor in the non-archimedean aspect, to understand the spectral side of a relative (e.g. Kuznetsov) trace formulae (e.g. \cite{BBM}, \cite{MV}).

As far as we know, in higher rank the averaging problems as described above have not been much dealt with over families with bounded analytic conductor in the archimedean aspect, due to unavailability of ``newvectors with required invariance'' at the archimedean place. In a previous work, jointly with Paul. D. Nelson, we gave an analytic analogue of the non-archimedean newvector theory at an archimedean place \cite[Theorem 1, Theorem 2]{JN}. Loosely speaking, we showed that for every nice enough representation $\pi$ (e.g. archimedean component of a cuspidal automorphic representation) of $\PGL_n(\R)$ there are nonzero vectors in $\pi$, which we call \emph{analytic newvectors}, which are invariant by some open ``approximate congruence subgroup'' analogous to $K_0(p^N)$ in the non-archimedean setting (see \eqref{defn-congruence-subgroup} for definition). 

The main contribution of the existence of analytic newvectors goes in the construction of the weight function $J_X(\pi)$ as in Theorem \ref{orthogonality-full-spectrum} so that its integral transform under the Kuznetsov trace formula is fairly easy to understand. Apart from the properties of $J_X$ stated in Theorem \ref{orthogonality-full-spectrum}, we also expect that $J_X(\pi)$ should decay rapidly if $C(\pi)\gg X^{1+\epsilon}$ i.e., $J_X$ behaves like an approximate projector on the automorphic representations with analytic conductor bounded by $X$, see \cite[Remark 5]{JN}. However, in this article we do not pursue in this direction and hope to come back in future.

\subsection*{Acknowledgement}
The author thanks his doctoral advisor Paul D. Nelson for several helpful discussions and guidance. The author also thanks Valentin Blomer and Kannan Soundararajan for many useful remarks, Jesse Thorner for several feedback on an earlier draft, and Farrell Brumley for encouragement. The author also acknowledges his doctoral school ETH Z\"urich and Max-Planck-Institute f\"ur Mathematik where most of the work has been done.

\section{Preliminary set-up}

\subsection{Basic notations}\label{basic-notations} 
Let $N<G$ be the maximal unipotent subgroup consisting of upper triangular matrices and $W$ be the Weyl group of $G$. For any subgroup $H<G$ we denote $[H]:=\Gamma\cap H\backslash H$. We define an additive character of $N$ by 
\begin{equation}\label{additive-char-unip}
    \psi_m(n(x))=e\left(\sum_{i=1}^{r-1}m_ix_{i,i+1}\right),\quad e(z):=e^{2\pi i z}, n(x)=(x_ij)_{i,j},
\end{equation}
for some $m\in \Z^{r-1}$. We abbreviate $\psi_{(1,\dots,1)}$ by $\psi$. We fix some Haar measures $dg$ and $dn$ on $G$ and $N$, respectively. Correspondingly, we fix right $G$-invariant quotient measures on $\X$ and $N\backslash G$, which will again, by abuse of notations, be denoted by $dg$. By $\delta(g)$ we will denote that the modular character of $G$ evaluated at $g$.

\subsection{Automorphic Forms}\label{automorphic-forms}
For details and general discussions we refer to \cite{G}. For any $\pi\in \hat{\X}_\gen$ and $\varphi\in \pi$ we denote its Whittaker functional by
\begin{equation}
    W_\varphi(g):=\int_{[N]}\varphi(xg)\overline{\psi(x)}dx.
\end{equation}
As $\pi$ is generic there is a $G$-equivariant isomorphism between $\pi$ and its Whittaker model $\W(\pi,\psi):=\{W_\varphi\mid \varphi\in\pi\}$, where $G$ acts on $\W(\pi,\psi)$ by right translation. 

We fix the usual $G$-invariant inner product on $\pi$  e.g., if $\pi$ is cuspidal it is induced from the inner product underlying $L^2(\X)$. A unitary structure on $\W(\pi,{\psi})$ can be given by (see \cite[Chapter 3]{J})
\begin{equation}\label{defn-inner-prod}
    \langle W_1,W_2\rangle_{\W(\pi,{\psi})}:=\int_{N_{r-1}\backslash \GL_{r-1}(\R)}W_1\left[\begin{pmatrix}h&\\&1\end{pmatrix}\right]\overline{W_2\left[\begin{pmatrix}h&\\&1\end{pmatrix}\right]}dh,
\end{equation}
for $W_1,W_2\in \W(\pi,{\psi})$, where $N_{r-1}$ is the maximal unipotent of the upper triangular matrices in $\GL_{r-1}(\R)$. We note that, by Schur's Lemma, there exists a positive constant $\ell(\pi)$ such that
\begin{equation}\label{defn-ell-pi}
    \|\varphi\|^2_{\pi}=\ell(\pi)\|W_\varphi\|^2_{\W(\pi,\psi)}.
\end{equation}
It is known that when $\pi$ is cuspidal then $\ell(\pi)\asymp L(1,\pi,\Ad)$ where the implied constant in $\asymp$ is absolute (coming from the residue of a maximal Eisenstein series at $1$, see \cite[p. $617$]{B3}).

We denote the $m$'th Fourier coefficient, for $m\in \Z^{r-1}$, attached to $\pi$ by $\lambda_\pi(m)$. So for $\varphi\in \pi$ and $m\in \N^{r-1}$ we can write
\begin{equation}\label{unipotent-integarl}
    \int_{[N]}\varphi(xg)\overline{\psi_m(x)}dx=\frac{\lambda_\pi(m)}{\delta^{1/2}(\tilde{m})}W_\varphi(\tilde{m}g),
\end{equation}
where 
$$\tilde{m}=\diag(m_1\dots m_{r-1}, m_1\dots m_{r-2}, \dots, m_1,1).$$
Here $\lambda_\pi(m)$ is normalized so that $\lambda_\pi(1,\dots,1)=1$ and the Ramanujan conjecture would imply that $\lambda_\pi(m)\ll_\epsilon |m|^\epsilon$. 

\subsection{L-functions and conductor}\label{l-function-condutor} One can attach a global $L$-function to a $\pi\in \hat{\X}_\gen$, denoted by $L(s,\pi)$, which can be given by a Dirichlet series in some right half plane by
$$L(s,\pi):=\sum_{n=1}^\infty\frac{\lambda_\pi(n,1,\dots,1)}{n^s},$$
and meromorphically continued to all $s\in \C$. Every such $L$-function satisfies a functional equation as follows
\begin{equation}\label{global-functional-equation}
    L(1/2+s,\pi)=\gamma_\infty(1/2+s,\pi)L(1/2-s,\tilde{\pi}),
\end{equation}
where $\tilde{\pi}$ is the contragredient of $\pi$ and $\gamma_\infty$ is the local gamma factor attached to the archimedean data attached to $\pi$. One can define the $\gamma$-factor by
$$\gamma_\infty(1/2+s,\pi):=\epsilon_\infty(\pi)\frac{\prod_{i=1}^r\Gamma_\R(1/2-s+\overline{\mu^\infty_i(\pi)})}{\prod_{i=1}^r\Gamma_\R(1/2+s+\mu^\infty_i(\pi))},$$
where $\epsilon_\infty$ is the $\epsilon$-factor which is of modulus 1, $\{\mu^\infty_i(\pi)\}$ are the archimedean Langlands parameters attached to $\pi$, and $\Gamma_\R(s):=\pi^{-s/2}\Gamma(s/2)$. One has $\epsilon_\infty(\tilde{\pi})=\overline{\epsilon_\infty(\pi)}$ and consequently
$$\gamma_\infty(1/2,\tilde{\pi})=\overline{\gamma_\infty(1/2,\pi)},\quad |\gamma_\infty(1/2,\pi)|=1.$$
Correspondingly, one can define the analytic conductor $C(\pi)$ of $\pi$ by
$$C(\pi):=\prod_{i=1}^r(1+|\mu^\infty_i(\pi)|).$$
It is known, e.g. by Stirling estimates, that
\begin{equation}\label{bound-gamma-factor}
    \gamma_\infty(1/2-s,\pi) \asymp_{\Re(s)} C(\pi\otimes |\det|^{\Im(s)})^{\Re(s)}\ll_{\Re(s)} C(\pi)^{\Re(s)}(1+|s|)^{r\Re(s)},
\end{equation}
as long as $s$ is away from a pole or zero of the $\gamma_\infty$.

\subsection{Kloosterman sum} We refer to \cite[Chapter 11]{G} and \cite{F} for detailed discussion of Kloosterman sum on $\GL(r)$. For a tuple of nonzero integers $c:=(c_1,\dots,c_{r-1})$ we denote the matrix $\diag(1/c_{r-1},c_{r-1}/c_{r-2},\dots, c_2/c_1,c_1)$ by $c^*$. For $w\in W$ an Weyl element let $\Gamma_w:=\Gamma\cap G_w$ where $G_w$ is the Bruhat cell of $G$ attached to $w$.

If, for $m,n\in \N^{r-1}$ and $w\in W$, the \emph{compatibility condition}
$$\psi_m(c^*wxw^{-1}c^{*-1})=\psi_n(x),\quad x\in w^{-1}Nw\cap N,$$
holds then the Kloosterman sum attached to $m,n$ and moduli $c$ is defined by
$$S_w(m,n;c):=\sum_{\substack{{\gamma\in \Gamma_N\backslash\Gamma_w/(w^{-1}\Gamma_N^tw\cap\Gamma_N)}\\{\gamma=b_1c^*wb_2}}}\psi_m(b_1)\psi_n(b_2),$$
where $\gamma=b_1c^*wb_2$ denotes its Bruhat decomposition (see \cite[Chapter 10]{G}). The following result is due to Friedberg, a proof can be found in \cite[p.175]{F}.

\begin{lemma}\label{kloosterman-support}
Let $w\in W$ be any Weyl element.
The compatibility condition in the definition of the Kloosterman sum is satisfied only if $w$ is of the form
$$\begin{pmatrix}&&I_{d_1}\\&\iddots&\\I_{d_k}&&\end{pmatrix},\quad d_1+\dots +d_k=r,$$
where $I_d$ is the identity matrix of rank $d$, i.e. $S_w(m,n;c)$ is nonzero only for this type of Weyl element.
\end{lemma}

\subsection{Bessel distribution} In this subsection we let $\pi$ to be an abstract generic irreducible representation of $G$. Let $\B(\pi)$ be an orthonormal basis of $\W(\pi,\psi)$. We define the Bessel distribution $J_\pi$ attached to $\pi$ by 
\begin{equation}\label{bessel-distribution}
    J_F(\pi):=\sum_{W\in \B(\pi)}\pi(F)W(1)\overline{W(1)},
\end{equation}
for some $F\in C^\infty_c(G)$. We refer to \cite{CPS1, BM, LO} and references there for general discussions about Bessel functions attached to generic representations.

\begin{lemma}\label{well-defined-bessel}
The RHS of \eqref{bessel-distribution} is well-defined and does not depend on the choice of the orthonormal basis $\B(\pi)$. Thus if $F$ is a self-convolution of some $f\in C_c^\infty(G)$, i.e. 
$$F(g)=\int_{G}f(gh)\overline{f(h)}dh,$$
then
$$J_F(\pi)=\sum_{W\in \B(\pi)}|\pi(f)W(1)|^2,$$
for $\B(\pi)$ some orthonormal basis of $\pi$.
\end{lemma}

\begin{proof}
The first assertion follows from \cite[Lemma 23.1, Lemma 23.3]{BM}. For the second assertion we see that the RHS of \eqref{bessel-distribution} is
$$\sum_{W\in \B(\pi)}\int_{G}\int_{G}f(gh)W(g)\overline{f(h)W(1)}dgdh=\sum_{W\in \B(\pi)}\int_{G}f(g)W(g)dg\int_{G}\overline{f(h)W(h)}dh,$$
after changing the basis to $\{\pi(h)W\}_{W\in\B(\pi)}$, and we conclude.
\end{proof}

\subsection{Pre-Kuznetsov formula} Now we will record a soft version of the Kuznetsov formula for $\GL(n)$, as in \cite{CPS1}. Our goal in this subsection is to prove Proposition \ref{kuznetsov}.

Recall that $d\mu_\aut$ is the automorphic Plancherel measure on $\hat{\X}$ compatible with the $G$-invariant measure of $\X$. We fix a test function $f\in C^\infty_c(G)$. We define $F$ to be the self-convolution of $f$ as in Lemma \ref{well-defined-bessel}.
We consider the sum 
$$\sum_{\gamma\in\Gamma}F(x_1^{-1}\gamma x_2), \quad x_1,x_2\in G.$$
The above sum as a function of $x_2$ is left $\Gamma$-invariant, hence lives in $L^2(\X)$. So we can spectrally decompose the sum as follows.
\begin{equation}\label{spectral-decomposition}
    \sum_{\gamma\in\Gamma}F(x_1^{-1}\gamma x_2)=\int_{\hat{\X}}\sum_{\varphi\in\B^o(\pi)}\frac{\overline{\pi(\bar{F})\varphi(x_1)}{\varphi(x_2)}}{\|\varphi\|^2_\pi}d\mu_{\aut}(\pi),
\end{equation}
where $\B^o(\pi)$ is an orthogonal basis of $\pi$. For generic $\pi$ we normalize $\varphi\in\pi$ so that its first Fourier coefficient is $1$.

We fix $m,n\in \N^{r-1}$ and integrate both sides of \eqref{spectral-decomposition} against $\psi_m(x_1)\overline{\psi_n(x_2)}$ over $x_1,x_2\in [N]$. As the non-generic part of the spectrum automatically vanishes, we obtain using \eqref{unipotent-integarl} and \eqref{defn-ell-pi} that
\begin{multline}\label{pre-kuznetsov}
    \int_{\hat{\X}_\gen}\ell(\pi)^{-1}\frac{\overline{\lambda_\pi(m)}\lambda_\pi(n)}{\delta^{1/2}(\tilde{m}\tilde{n})}\sum_{W\in\B(\pi)}\overline{\pi(\bar{F})W(\tilde{m})}W(\tilde{n})d\mu_{\mathrm{aut}}(\pi)\\=\sum_{\gamma\in\Gamma}\int_{[N]^2}F(x_1^{-1}\gamma x_2)\psi_m(x_1)\overline{\psi_n(x_2)}dx_1dx_2.
\end{multline}
Here we have identified $\pi$ with its Whittaker model $\W(\pi)$ for generic $\pi$ and $\B(\pi)$ is an orthonormal basis of $\W(\pi)$.

We replace $F$ in \eqref{pre-kuznetsov} by $F_{m,n}$ where
$$F_{m,n}(g):=F(\tilde{m}g\tilde{n}^{-1}).$$
Using Lemma \ref{well-defined-bessel} we obtain 
$$\sum_{W\in\B(\pi)}\overline{\pi(\bar{F}_{m,n})W(\tilde{m})}W(\tilde{n})=J_{\bar{F}}(\pi).$$

We replace $F$ by $F_{m,n}$ in the RHS of \eqref{pre-kuznetsov} as well and simplify. Doing a folding-unfolding over the $\gamma$-sum and $x_1$-integral we obtain the RHS is
$$\sum_{\gamma\in \Gamma_N\backslash\Gamma}\int_{N}\int_{[N]}F(\tilde{m}x_1\gamma x_2\tilde{n}^{-1})\overline{\psi_m(x_1)\psi_n(x_2)}dx_1dx_2.$$
Note that the identity term in the above sum simplifies to
$$\int_NF(\tilde{m}x_1\tilde{n}^{-1})\overline{\psi_m(x_1)}dx_1\int_{[N]}\overline{\psi_n(x_2)}\psi_m(x_2)dx_2=\frac{\delta_{m=n}}{\delta(\tilde{m})}\int_N F(x)\overline{\psi(x)}dx.$$
On the other hand, we do a Bruhat decomposition of $\Gamma- \Gamma_N$. For $\gamma\in \Gamma_w$ which is the Bruhat cell attached to $w\in W$ we write $x_2$-integral over $[N]$ as a product of $x_{2,1}$-integral over $[N_w]$ and $x_{2,2}$-integral over $[\bar{N}_w]$, where
$$\bar{N}_w:=w^{-1}N w\cap N,\quad {N}_w:=w^{-1}N^tw\cap N.$$
\begin{multline*}
    \sum_{\gamma\in \Gamma_N\backslash(\Gamma-\Gamma_N)}\int_{N}\int_{[N]}F(\tilde{m}x_1\gamma x_2\tilde{n}^{-1})\overline{\psi_m(x_1)\psi_n(x_2)}dx_1dx_2\\
    =\sum_{1\neq w\in W}\sum_{c\in \Z^{r-1}_{\neq 0}}\sum_{\substack{{\gamma\in \Gamma_N\backslash\Gamma_w/\Gamma_{{N}_w}}\\{\gamma=b_1c^*wb_2}}}\sum_{\theta\in \Gamma_{{N}_w}}\int_N\int_{[N_w]}\int_{[\bar{N}_w]}F(\tilde{m}x_1b_1c^*wb_{2}\theta x_{2,2}x_{2,1}\tilde{n}^{-1})\\
    \times\overline{\psi_m(x_1)\psi_n(x_{2,1})\psi_n(x_{2,2})}dx_{2,2}dx_{2,1}dx_1.
\end{multline*}
Again unfolding over the $\theta$-sum and $x_{2,2}$-integral, changing of variables, and recalling the definition of the Kloosterman sum we obtain the summand above corresponding to $w\in W $ and $c\in \Z^{r-1}_{\neq 0}$ is
\begin{equation*}
    S_w(m,n;c)\int_N\int_{N_w}\int_{[\bar{N}_w]}F(\tilde{m}x_1c^*wx_{2,1}x_{2,2}\tilde{n}^{-1})\overline{\psi_m(x_1)\psi_n(x_{2,2})\psi_n(x_{2,1})}dx_{2,1}dx_{2,2}dx_1.
\end{equation*}
Now writing $x_{2,1}=x_{2,1}(c^*w)^{-1}(c^*w)\in N$ and using the compatibility condition of the Kloosterman sum we obtain the above equals to
\begin{equation*}
    S_w(m,n;c)\int_N\int_{N_w}F(\tilde{m}x_1c^*wx_{2}\tilde{n}^{-1})\overline{\psi_m(x_1)\psi_n(x_{2})}dx_2dx_1.
\end{equation*}
Finally, changing variables $x_1\mapsto \tilde{m}x_1\tilde{m}^{-1}$ and $x_2\mapsto\tilde{n}x_2\tilde{n}^{-1}$
we obtain the RHS of \eqref{pre-kuznetsov} is
\begin{equation*}
    \sum_{1\neq w\in W}\sum_{c\in \Z^{r-1}_{\neq 0}}\frac{S_w(m,n;c)}{\delta(\tilde{m})\delta_w(\tilde{n})}\int_N\int_{N_w}F(x_1\tilde{m}cw\tilde{n}^{-1}x_2)\overline{\psi(x_1)\psi(x_{2})}dx_2dx_1.
\end{equation*}
Here $\delta_w(\tilde{n})$ is the Jacobian arising from the change of variable $x_2\mapsto\tilde{n}x_2\tilde{n}^{-1}$.
We define, for $F\in C_c^\infty(G)$, the function
\begin{equation}\label{psi-average}
    W_F(g):=\int_N F(xg)\overline{\psi(x)}dx,
\end{equation}
lies in $C^\infty_c(N\backslash G,\psi)$. It can be checked that the RHS of \eqref{psi-average} is absolutely convergent.
Thus we obtain a version of the Kuznetsov trace formula.
\begin{prop}\label{kuznetsov}
Let $F\in C^\infty_c(G)$ and $m,n\in \N^{r-1}$. Then
\begin{multline*}
    \int_{\hat{\X}_\gen}\overline{\lambda_\pi(m)}\lambda_\pi(n)\frac{J_{\bar{F}}(\pi)}{\ell(\pi)}d\mu_\aut(\pi)=\delta_{m=n}W_F(1)\\+\sum_{1\neq w\in W}\frac{\delta^{1/2}(\tilde{n})}{\delta^{1/2}(\tilde{m})\delta_w(\tilde{n})}\sum_{c\in \Z^{r-1}_{\neq 0}}S_w(m,n;c)\int_{N_w}W_F(\tilde{m}c^*w\tilde{n}^{-1}x)\overline{\psi(x)}dx,
\end{multline*}
where $W_F$ and $J_F$ are as in \eqref{psi-average} and \eqref{bessel-distribution}, respectively.
\end{prop}

\section{Proof of Theorems \ref{orthogonality-full-spectrum}, \ref{large-sieve} and \ref{second-moment}}

\subsection{Auxiliary lemmata and proof of Theorems \ref{orthogonality-full-spectrum} and \ref{large-sieve}}
We start by recalling an \emph{approximate subgroup} $K_0(X,\tau)$ of $G$, as in \cite{JN}, where $X>1$ tending off to infinity and $\tau>0$ is sufficiently small but fixed. From here on we will, as usual in analytic number theory, not distinguish between $\tau$ and a fixed constant multiple for it.
\begin{equation}\label{defn-congruence-subgroup}
    K_0(X,\tau):=
    \mathrm{Im}_{G}\left\{\begin{pmatrix}
    a&b\\c&d
    \end{pmatrix}\in\mathrm{SL}_{r}(\R)\middle|
    a \in \GL_{r-1}(\R),
    d \in \GL_1(\R),
    \begin{aligned}
    &|a - 1_{r-1}| < \tau,
    \quad
    |b|<\tau, \\
    &|c|<\frac{\tau}{X},
    \quad
    |d-1|<\tau
    \end{aligned}
    \right\}.
\end{equation}
Here, various $|.|$ denote arbitrary fixed
norms on the corresponding spaces of matrices.
Let $f_X$ be a smoothened $L^1$-normalized non-negative characteristic function of $f_X$ and $F_X$ be the self-convolution of $f_X$ as in Lemma \ref{well-defined-bessel}. We abbreviate $J_{F_X}$ and $W_{F_X}$ as $J_X$ and $W_X$, respectively (see \eqref{bessel-distribution} and \eqref{psi-average}). 

\begin{lemma}\label{spectral-weight-property}
The function $J_X$ as in \eqref{bessel-distribution} is non-negative. For $\pi\in \hat{\X}_\gen$ cuspidal if the analytic conductor $C(\pi)$ is smaller than $X$ then $J_X(\pi)\gg1$.
\end{lemma}

\begin{proof}
Non-negativity of $J_X(\pi)$ follows from Lemma \ref{well-defined-bessel}. Let $\pi$ be cuspidal. We recall \cite[Theorem 7]{JN}\footnote{Actually implication of \cite[Theorem 7]{JN}, as generic irreducible unitary cuspidal automorphic representations of $\PGL_r(\R)$ are $\theta$-tempered for $\theta<1/2-1/r^2+1$ \cite{MS}.}. For each $\epsilon>0$ there exists a $\tau>0$ such that for each generic irreducible unitary cuspidal automorphic representation $\pi$ of $G$ there exists an \emph{analytic newvector} $W\in \pi$ with
\begin{itemize}
    \item $|W(g)-W(1)|< \epsilon$ for all $g\in K_0(C(\pi),\tau)$,
    \item and $W(1)\gg 1$.
\end{itemize}
Thus if $C(\pi)<X$, i.e. $K_0(C(\pi),\tau)\supseteq K_0(X,\tau)$,
$$|\pi(f_X)W(1)-W(1)|=|\int_{K_0(X,\tau)} f_X(g)(W(g)-W(1))dg|\le \epsilon,$$
hence $\pi(f_X)W(1)\gg 1$.

We choose an orthonormal basis $\B(\pi)$ containing the above analytic newvector $W$ of $\pi$. Thus by dropping all but the summand corresponding to $W$ in the expression of $J_X(\pi)$ as in Lemma \ref{well-defined-bessel} we conclude that
$J_X(\pi)\gg 1$.
\end{proof}

\begin{lemma}\label{support-gamma}
Let $X>1$ be large enough and $\tau>0$ be sufficiently small but fixed. Let $m,n\in \N^{r-1}$ and $x_1,x_2\in [N]$. Let $\gamma\in \Gamma$ such that $\tilde{m}x_1\gamma x_2\tilde{n}^{-1}\in K_0(X,\tau)$. If 
$$\min(n_{r-1}^{r-1}\dots n_1m_1^{r-1}\dots m_{r-1},m_{r-1}^{r-1}\dots m_1 n_1^{r-1}\dots n_{r-1})\ll X^r$$
with sufficiently small implied constant (depending at most on $\tau$) then the last row of $\gamma$ is $(0,\dots,0,1)$.
\end{lemma}

\begin{proof}
We write $\gamma=\begin{pmatrix}a&b\\c&d\end{pmatrix}$ in $r-1,1$ block decomposition, i.e. $(c,d)$ is the bottom row of $\gamma$. Let $N:=(n_{r-1}^{r-1}\dots n_1m_1^{r-1}\dots m_{r-1})^{1/r}\ll X$ with sufficiently small constant. From the definition of $K_0(X,\tau)$ in \eqref{defn-congruence-subgroup} we obtain that the last row of $\det(\tilde{m}\tilde{n}^{-1})^{1/r}\tilde{m}x_1\gamma x_2\tilde{n}^{-1}$
$$|\det(\tilde{m}\tilde{n}^{-1})^{1/r}(Xc,d)x_2\tilde{n}^{-1}-(0,1)| < \tau.$$
Thus $|c|\ll \tau N/X\ll \tau$ with a sufficiently small implied constant. As coordinates of $c$ are integers for sufficiently small $\tau$ we get $c=0$, consequently, $d=1$, as $\gamma$ lies in $\PGL_r(\Z)$.  

However, $\tilde{n}x_2^{-1}\gamma^{-1}x_1^{-1}\tilde{m}^{-1}\in K_0(X,\tau)$. Interchanging $m$ and $n$ in the above argument we obtain that if $m_{r-1}^{r-1}\dots m_1 n_1^{r-1}\dots n_{r-1}\ll X^r$ with sufficiently small implied constant the last row of $\gamma^{-1}$ is $(0,\dots,0,1)$, hence the same holds for $\gamma$ and we conclude the proof.
\end{proof}

\begin{lemma}\label{error-term-vanish}
Let $m,n,x_1,x_2$ be as in Lemma \ref{support-gamma}. Then
$$\sum_{\gamma\in\Gamma-\Gamma_N}\int_{[N]^2}F_X(\tilde{m}x_1\gamma x_2\tilde{n}^{-1})\overline{\psi_m(x_1)\psi_n(x_2)}dx_1dx_2=0,$$
where $F_X$ is as defined after \eqref{defn-congruence-subgroup}.
\end{lemma}

\begin{proof}
From (the proof of) Proposition \ref{kuznetsov} we obtain that
\begin{multline*}
    \sum_{\gamma\in\Gamma-\Gamma_N}\int_{[N]^2}F_X(\tilde{m}x_1\gamma x_2\tilde{n}^{-1})\overline{\psi_m(x_1)\psi_n(x_2)}dx_1dx_2\\
    =\sum_{1\neq w\in W}\frac{\delta^{1/2}(\tilde{n})}{\delta^{1/2}(\tilde{m})\delta_w(\tilde{n})}\sum_{c\in \Z^{r-1}_{\neq 0}}S_w(m,n;c)\int_{N_w}W_X(\tilde{m}c^*w\tilde{n}^{-1}x)\overline{\psi(x)}dx.
\end{multline*}
Hence, from Lemma \ref{kloosterman-support} we conclude that it is enough to consider $\gamma\in \Gamma-\Gamma_N$ in the Bruhat cell attached to the Weyl elements $w$ of the form $\begin{pmatrix}&&I_{d_1}\\&\iddots&\\I_{d_k}&&\end{pmatrix}$ with $d_k<r$, which implies that the last row of $\gamma$ is not of the form $(0,\dots,0,1)$. Therefore, from Lemma \ref{support-gamma} we obtain the support condition that $\tilde{m}x_1\gamma x_2\tilde{n}^{-1}\notin K_0(X,\tau)$ for small enough $\tau$. We conclude by noting that the support of $F_X$ is $K_0(X,\tau)$, which follows from the definition of $F_X$. 
\end{proof}

\begin{proof}[Proof of Theorem \ref{orthogonality-full-spectrum}]
Proposition \ref{kuznetsov} and Lemma \ref{error-term-vanish} imply that
$$\int_{\hat{\X}_\gen}\frac{\lambda_\pi(m)\overline{\lambda_\pi(n)}}{\ell(\pi)}J_X(\pi)d\mu_\aut(\pi)=\delta_{m=n}W_X(1).$$
Choosing $m=n=(1,\dots,1)$ in the Proposition \ref{kuznetsov} we obtain
$$\int_{\hat{X}_\gen}J_X(\pi)d\mu_\aut(\pi)=W_X(1).$$
Finally noting that $$W_X(1)\asymp\mathrm{vol}(K_0(X,\tau))^{-1}\asymp X^{r-1},$$
we conclude the proof.
\end{proof}

\begin{proof}[Proof of Theorem \ref{large-sieve}]
Using Lemma \ref{spectral-weight-property}, \eqref{defn-ell-pi}, and artificially adding the continuous spectrum we write
$$\sum_{\substack{{C(\pi)<X}\\{\pi\,\,\mathrm{cuspidal}}}}{L(1,\pi,\Ad)^{-1}}\left|\sum_{N\ll X}\alpha(n)\lambda_\pi(n)\right|^2\ll \int_{\hat{\X}_\gen}\left|\sum_{N\ll X}\alpha(n)\lambda_\pi(n)\right|^2\frac{J_X(\pi)}{\ell(\pi)}d\mu_\aut(\pi).$$
We open the square above, let $M:=m_1,\dots m_{r-1}$ for $m\in \N^{r-1}$ to obtain the RHS is
$$\sum_{N,M\ll X}\overline{\alpha(m)}\alpha(n)\int_{\hat{\X}_\gen}\overline{\lambda_{\pi}(m)}\lambda_\pi(n)\frac{J_X(\pi)}{\ell(\pi)}d\mu_{\aut}(\pi).$$
Note that, $N,M\ll X$ with sufficiently small implied constant yields that
$$\min(n_{r-1}^{r-1}\dots n_1m_1^{r-1}\dots m_{r-1},m_{r-1}^{r-1}\dots m_1 n_1^{r-1}\dots n_{r-1})\ll X^r$$
with sufficiently small implied constant. If not, there would be a sequence of $m,n$ such that
$$(m_1\dots m_{r-1} n_1\dots n_{r-1})^r\gg X^{2r},$$
with large implied constants, contradicting the assumptions on $M,N$. 

Thus we can apply Theorem \ref{orthogonality-full-spectrum} to obtain that
$$\sum_{\substack{{C(\pi)<X}\\{\pi\,\,\mathrm{cuspidal}}}}{L(1,\pi,\Ad)^{-1}}\left|\sum_{N\ll X}\alpha(n)\lambda_\pi(n)\right|^2\ll\sum_{N,M\ll X}\overline{\alpha(m)}\alpha(n)\delta_{m=n}W_X(1).$$
We conclude the proof recalling that $W_X(1)\asymp X^{r-1}$ from the proof of Theorem \ref{orthogonality-full-spectrum}.
\end{proof}

\subsection{Proof of Theorem \ref{second-moment}}
In this subsection we change the notation a bit. For $n\in\Z$ we will denote the Fourier coefficient $\lambda_\pi(n,1,\dots,1)$ by $\lambda_\pi(n)$. Following lemma is a standard exercise using the approximate functional equation of the $L$-value. We include the argument for the sake of completeness.

\begin{lemma}\label{approximate-functional-equation}
Let $\pi$ be cuspidal with $C(\pi)<X$. Then,
\begin{equation*}
    |L(1/2,\pi)|^2\ll_\epsilon X^{\epsilon}\int_{|t|\ll X^\epsilon}\left|\sum_{n\ll X^{1/2+\epsilon}}\frac{\lambda_\pi(n)}{n^{1/2+\epsilon+it}}\right|^2dt + O(X^{-A}),
\end{equation*}
for all large $A>0$.
\end{lemma}

\begin{proof}
We start by proving approximate functional equation as in \cite{H}. We define
$$H_X(s,\pi):=\frac{1}{2}X^{s/2}+\frac{1}{2}X^{-s/2}\gamma_\infty(1/2,\pi)\gamma_\infty(1/2-s,\tilde{\pi}).$$
Note that, $H_X(s,\pi)$ is holomorphic as $\Re(s)>0$ and $H_X(0,\pi)=1$.
From \eqref{global-functional-equation} we can write the global functional equation as
$$L(1/2+s,\pi)H_X(s,\pi)=\gamma_\infty(1/2,\pi)L(1/2-s,\tilde{\pi})H_X(-s,\tilde{\pi}).$$
The estimate in \eqref{bound-gamma-factor} implies that
$$H_X(s,\pi)\ll_{\Re(s)} X^{\Re(s)/2}\left(1+\left({C(\pi)}/{X}\right)^{\Re(s)}(1+|s|)^{O(1)}\right),$$
We choose an entire function $h$ with $h(0)=1$ and $h(s)=h(-s)$. Note for $\pi$ cuspidal that $L(1/2+s,\pi)H_X(s,\pi)$ is entire. Then by Cauchy's theorem, and applying $s\mapsto-s$ and the above version of the global functional equation in the second equality we get
\begin{align*}
    L(1/2,\pi)
    &=\int_{(1)}L(1/2+s,\pi)H_X(s,\pi)h(s)\frac{ds}{s}-\int_{(-1)}L(1/2+s,\pi)H_X(s,\pi)h(s)\frac{ds}{s}\\
    &=\int_{(1)}L(1/2+s,\pi)H_X(s,\pi)h(s)\frac{ds}{s}+\gamma_\infty(1/2,\pi)\int_{(1)}L(1/2+s,\tilde{\pi})H_X(s,\tilde{\pi})h(s)\frac{ds}{s}\\
    &=\sum_{n}\frac{\lambda_\pi(n)}{\sqrt{n}}V_1(n)+\gamma_{\infty}(1/2,\pi)\sum_{n}\frac{\overline{\lambda_\pi(n)}}{\sqrt{n}}V_2(n),
\end{align*}
where in the last equality we have used the Dirichlet series of $L(s,\pi)$ as in \S\ref{l-function-condutor}. Here
$$V_1(n)=V_1(n;X,\pi):=\int_{(1)}n^{-s}H_X(s,\pi)h(s)\frac{ds}{s},$$
and $V_2$ has a similar definition. 

Let $C(\pi)<X$ from now on. Thus $H_X(s,\pi)\ll X^{\Re(s)/2}(1+|s|)^{O(1)}$ We first claim that the above sums can be truncated at $n\le X^{1/2+\epsilon}$ with an error $O(X^{-A})$ for all large $A$. To see that we shift the contour in the defining integral of $V_1$ to $K$ for some large $K>0$. Thus we obtain 
$$V_1(n)\ll_K n^{-K}X^{K/2}.$$
Consequently, using $\lambda_\pi(n)\ll n^{c}$ for some fixed $c$ (follow from e.g. \cite{JS}),
$$\sum_{n> X^{1/2+\epsilon}}\frac{\lambda_\pi(n)}{\sqrt{n}}V_1(n)\ll_K X^{K/2}\sum_{n>X^{1/2+\epsilon}}n^{-1/2-K+c}\ll_B X^{-B},$$
for all large $B$. 

We also note that, for $\epsilon>0$ small,
\begin{align*}
    V_1(n)
    &=\int_{(\epsilon)}n^{-s}H_X(s,\pi)h(s)\frac{ds}{s}\\
    &=\int_{|t|< X^{\epsilon}}n^{-\epsilon-it}H_X(\epsilon+it,\pi)h(\epsilon+it)\frac{dt}{\epsilon+it}+O\left(n^{-\epsilon} X^{\epsilon/2}\int_{|t|\ge X^\epsilon}|t|^{-K}dt\right)\\
    &=\int_{|t|< X^{\epsilon}}n^{-\epsilon-it}H_X(\epsilon+it,\pi)h(\epsilon+it)\frac{dt}{\epsilon+it}+O(n^{-\epsilon}X^{-B}).
\end{align*}
But,
\begin{align*}
    \sum_{n\ll X^{1/2+\epsilon}}\frac{\lambda_\pi(n)}{n^{1/2+\epsilon}}\ll X^{(1/2+\epsilon)(1/2+c-\epsilon)}.
\end{align*}
Thus we obtain that
$$\sum_{n}\frac{\lambda_\pi(n)}{\sqrt{n}}V_1(n)=\int_{|t|\ll X^{\epsilon}}\sum_{n\ll X^{1/2+\epsilon}}\frac{\lambda_\pi(n)}{n^{1/2+\epsilon+it}}H_X(\epsilon+it,\pi)\frac{h(\epsilon+it)}{\epsilon+it}dt+O(X^{-B}),$$
for all large $B$.

We can do similar analysis for $V_2$ summand. Thus using that $|\gamma_\infty(1/2,\pi)|=1$ we get
\begin{equation*}
    |L(1/2,\pi)|^2\ll \left|\int_{|t|\ll X^{\epsilon}}\sum_{n\ll X^{1/2+\epsilon}}\frac{\lambda_\pi(n)}{n^{1/2+\epsilon+it}}H_X(\epsilon+it,\pi)\frac{h(\epsilon+it)}{\epsilon+it}dt\right|^2+O(X^{-B}).
\end{equation*}
Finally, note that
$$\int_{|t|\ll X^\epsilon}\left|H_X(\epsilon+it,\pi)\frac{h(\epsilon+it)}{\epsilon+it}\right|^2dt\ll X^\epsilon.$$
Thus doing a Cauchy-Schwarz in the first term of the estimate of $|L(1/2,\pi)|^2$ above ,we conclude the proof.
\end{proof}

\begin{proof}[Proof of Theorem \ref{second-moment}]
Lemma \ref{approximate-functional-equation} and \cite[Theorem 3]{JN} imply that
$$\sum_{\substack{{C(\pi)<X}\\{\pi\,\,\mathrm{cuspidal}}}}\frac{|L(1/2,\pi)|^2}{L(1,\pi,\Ad)}\ll_\epsilon X^\epsilon\int_{|t|\ll X^\epsilon}\sum_{\substack{{C(\pi)<X}\\{\pi\,\,\mathrm{cuspidal}}}}L(1,\Ad,\pi)^{-1}\left|\sum_{n\ll X^{1/2+\epsilon}}\frac{\lambda_\pi(n)}{n^{1/2+\epsilon+it}}\right|^2dt+O(X^{-A}).$$
Using Lemma \ref{large-sieve} with $\alpha(n,1,\dots,1)=n^{-1/2-\epsilon-it}$, for $n\ll X^{1/2+\epsilon}$ and $\alpha=0$, otherwise, we obtain that the first sum in the RHS is bounded by $X^{r-1+\epsilon}$ and thus we conclude.
\end{proof}

\section{Proof of Theorems \ref{sato-tate} and \ref{density}}

\subsection{Fourier coefficients for \texorpdfstring{$\PGL(r)$}{}} Multiplicative properties of the Fourier coefficients of automorphic forms for $\PGL_r(\Z)$ can be understood in terms of the standard character theory of $\mathrm{SU}(r)$ by the work of Shintani and Casselman--Shalika.
It is known that the irreducible representations of $\SU(r)$ are in bijection with the dominant elements (i.e. elements with non-increasing coordinates) of $\Z^{r}/\Z(1,\dots,1)$ which are in bijection with the elements of $\Z_{\ge 0}^{r-1}$ by the map
$$\iota:\Z^r/\Z(1,\dots,1)\ni \alpha\mapsto (\alpha_1-\alpha_2,\dots,\alpha_{r-1}-\alpha_{r})\in\Z^{r-1}_{\ge 0}.$$
For $\alpha\in \Z^{r-1}_{\ge 0}$, let $\chi_\alpha$ be the character of the highest weight representation of $\SU(r)$ attached to $\alpha$, given by a Schur polynomial. The following Lemma, which is a theorem by Shintani and Casselman--Shalika \cite{Sh,CS}, gives a formula of the Fourier coefficients of $\pi$ in terms of the Satake parameters.

\begin{lemma}\label{casselman-shalika}
Let $\pi\in \hat{\X}_\gen$ with Satake parameters $\mu^p(\pi)$ at a finite prime $p$. Then
$$\chi_\alpha(\diag(p^{\mu^p(\pi)}))=\lambda_\pi(p^{\iota(\alpha)}),$$
where $\lambda_\pi$ is defined in \eqref{unipotent-integarl}.
\end{lemma}

Lemma \ref{casselman-shalika} helps to linearize the monomials of Fourier coefficients. By decomposing the tensor product of the highest weight representations in a direct sum of irreducible representations attached to various $\alpha$ we can compute that for any multi-indices $\alpha^{(1)},\dots,\alpha^{(k)}\in \Z^r/\Z(1,\dots,1)$ and nonnegative integers $s_1,\dots,s_k$
we may write
$$\prod_{i}\chi_{\alpha^{i)}}^{s_i}=\sum_\beta c_\beta\chi_{\beta},$$
for some coefficients $c_\beta$. Correspondingly, using Lemma \ref{casselman-shalika}, we get
\begin{equation}\label{linearize-hecke-eigenvalue}
    \prod_i\lambda_\pi(p^{\iota(\alpha_i)})^{s_i}=\sum_\beta c_\beta\lambda(p^{\iota(\beta)}),
\end{equation}
where the sums in the RHS are finite.

\begin{proof}[Proof of Theorem \ref{sato-tate}] 
We proceed exactly as in \cite{Zh}. Using Lemma \ref{casselman-shalika} we can embed the Satake parameters into the space of Fourier coefficients by mapping under elementary symmetric polynomials: $$\rho:p^{\mu^p(\pi)}\mapsto (e_1(p^{\mu^p(\pi)}),\dots, e_{r-1}(p^{\mu^p(\pi)})),$$
where $e_j$ is the $j$'th elementary symmetric polynomial. It can be checked that
$$e_j(p^{\mu^p(\pi)})=\lambda_\pi(\underbrace{p,\dots,p}_{j \text{ times}},1\dots,1).$$
Thus we may identify the space of continuous functions on $T_1/W$ as the same on the image of $\rho$ as a compact subset of $\C^{r-1}$. Hence, to prove Theorem \ref{sato-tate}, it suffices to prove that for any continuous $\tilde{f}$ (e.g. composing $f$ with $\rho^{-1}$) on the image of $\rho$ one has
\begin{equation}\label{reduced-sato-tate}
    \lim_{X\to\infty}\frac{\int_{\hat{\X}_\gen}\tilde{f}(e_1(p^{\mu^p(\pi)}),\dots, e_{r-1}(p^{\mu^p(\pi)}))\frac{J_X(\pi)}{\ell(\pi)}d\mu_\aut(\pi)}{\int_{\hat{\X}_\gen}\frac{J_X(\pi)}{\ell(\pi)}d\mu_\aut(\pi)}\to\int_{T_0/W}\tilde{f}\circ \rho(z)d\mu_{\mathrm{st}}(z),
\end{equation}
We can approximate $\tilde{f}$, by Stone--Weierstrauss, using polynomials in $z:=(z_1,\dots,z_{r-1})$ and $\bar{z}$. Thus by linearity it is enough to prove \eqref{reduced-sato-tate} for $\tilde{f}=z^{\alpha}\bar{z}^{\beta}$ for some multi-indices $\alpha,\beta$. 

Now note that for any multi-indices $\alpha,\beta\in \Z^r/\Z(1,\dots,1)$, from Theorem \ref{orthogonality-full-spectrum} we obtain
\begin{equation*}
    \lim_{X\to\infty}\frac{\int_{\hat{\X}_\gen}\lambda_\pi(p^{\iota(\alpha)})\overline{\lambda_\pi(p^{\iota(\beta)})}\frac{J_X(\pi)}{\ell(\pi)}d\mu_\aut(\pi)}{\int_{\hat{\X}_\gen}\frac{J_X(\pi)}{\ell(\pi)}d\mu_\aut(\pi)}=\delta_{\iota(\alpha)=\iota(\beta)}=\int_{T_0/W}\chi_\alpha\overline{\chi_\beta}(z)d\mu_{\mathrm{st}}(z),
\end{equation*}
where the last equality follows by Schur's orthogonality of characters. Finally, using \eqref{linearize-hecke-eigenvalue} (and the corresponding relation for $\chi_\alpha$) we conclude \eqref{reduced-sato-tate}.
\end{proof}

\begin{proof}[Proof of Theorem \ref{density}]
We follow the proof of \cite[Theorem 1]{B2}. From \cite[Lemma 4]{B2} we get that for $n>r$ one has
\begin{equation}\label{amplifier}
    \sum_{j=0}^{r-1}|\lambda_\pi(p^{n-j},1,\dots,1)|^2\ge (2p^{\theta^p(\pi)})^{2(1-r)}p^{2n\theta^p(\pi)}.
\end{equation}
We choose $n$ large enough such that $p^n\asymp X$ with small enough implied absolute constant (admissible by Theorem \ref{orthogonality-full-spectrum}). Then for $\theta^p(\pi)>\eta$ we obtain from \eqref{amplifier} that
$$\sum_{j=0}^{r-1}|\lambda_\pi(p^{n-j},1,\dots,1)|^2\gg X^{2\eta}.$$
On the other hand,
$$|\{\pi\,\,\mathrm{cuspidal}, C(\pi)<X\mid \theta^p(\pi)> \eta\}|\ll X^{-2\eta}\sum_{\substack{{C(\pi)<X}\\{\pi\,\,\mathrm{cuspidal}}}}\sum_{j=0}^{r-1}|\lambda_\pi(p^{n-j},1,\dots,1)|^2.$$
Using the fact that (see \cite{Li})
\begin{equation}\label{bound-l-value-at-1}
    L(1,\pi,\Ad)\ll C(\pi)^\epsilon,
\end{equation}
adding the similar contributions from continuous spectrum, and using Lemma \ref{spectral-weight-property} we obtain that
$$|\{\pi\,\,\mathrm{cuspidal}, C(\pi)<X\mid \theta^p(\pi)> \eta\}|\ll_\epsilon X^{-2\eta+\epsilon}\sum_{j=0}^{r-1}\int_{\hat{\X}_\gen}|\lambda_\pi(p^{n-j},1,\dots,1)|^2\frac{J_X(\pi)}{\ell(\pi)}d\mu_\aut(\pi).$$
From Theorem \ref{orthogonality-full-spectrum} we conclude that the RHS is $\ll_\epsilon X^{r-1-2\eta+\epsilon}$.
\end{proof}

\section{Proof of Theorem \ref{low-lying-zero}}
We recall Conjecture \ref{conj-langlands} and Conjecture \ref{conj-cond-size}. Throughout this section we will assume these two conjectures. These are also assumptions in the statement of Theorem \ref{low-lying-zero}.

\subsection{Explicit formula}
Recall that ${^{L}{G}}=\SL_r(\C)$ and the $L$-group homomorphism $\rho:\SL_r(\C)\to \GL_d(\C)$ coming from a $d$-dimensional ($d>1$) irreducible representation. We abbreviate $\rho_*\pi$ as $\Pi$ for $\pi\in\hat{\X}_\gen$. Let $\{\mu^p(\Pi)\}_{i=1}^d$ be the local Satake (resp. Langlands) parameters for $p<\infty$ (resp. $p=\infty$).

We start by recalling the functional equation of $L(s,\Pi)$ as in \S\ref{l-function-condutor}. We write the global $L$-function
$$\Lambda(s,\Pi)=L(s,\Pi)L_{\infty}(s,\Pi),$$
where $L_\infty(s,\Pi):=\prod_{i=1}^d\Gamma_\R(1/2+s+\mu^\infty_i(\Pi))$. Then \eqref{global-functional-equation} reads as 
$$\Lambda(s,\Pi)=\epsilon_\infty(\Pi)\Lambda(1-s,\tilde{\Pi}).$$
Let $\psi_0$ be an entire function with rapid decay in the vertical strips. 
Using Cauchy's argument principle and the global functional equation we can write 
\begin{equation}\label{cauchy-argument-principle}
    \sum_{\gamma\text{ zeros}}\psi_0(\gamma)-\sum_{\tau\text{ poles}}\psi_0(\tau)=   \int_{(2)}\psi_0(s)\frac{\Lambda'}{\Lambda}(s,\Pi)ds+\int_{(2)}\psi_0(1-s)\frac{\Lambda'}{\Lambda}(s,\tilde{\Pi})ds,
\end{equation}
where zeros and poles are of $\Lambda(.,\Pi)$ in the critical strip, i.e. $\Re(s)\in [0,1]$, counted with multiplicity. 

Let $\Re(s)=2$. We employ the Euler product
$$L(s,\Pi)=\prod_{p}\prod_{j=1}^d(1-p^{\mu_j^p(\Pi)-s})^{-1},$$
to obtain that
$$\frac{\Lambda'}{\Lambda}(s,\Pi)=\sum_{j=1}^d\frac{\Gamma'_\R}{\Gamma_\R}(s+\mu_j^\infty(\Pi))-\sum_{p<\infty}\log p\sum_{j=1}^d\sum_{k=1}^\infty p^{k(\mu_j^p(\Pi)-s)}.$$
We exchange $\int_{(2)}$ and $\sum_{p<\infty}$ which is justified by the absolute convergence of the Euler product and the rapid decay of $\psi_0$ in the vertical direction. We then shift each contour to the vertical line $\Re(s)=1/2$. 
Thus
we obtain the RHS of \eqref{cauchy-argument-principle}
\begin{multline*}
\int_{(1/2)}\psi_0(s)\left(\sum_{j=1}^d\frac{\Gamma'_\R}{\Gamma_\R}(s+\mu_j^\infty(\Pi))-\sum_{p<\infty}\log p\sum_{j=1}^d\sum_{k=1}^\infty p^{k(\mu_j^p(\Pi)-s)}\right)ds\\+\int_{(1/2)}\psi_0(1-s)\left(\sum_{j=1}^d\frac{\Gamma'_\R}{\Gamma_\R}(s+\overline{\mu_j^\infty(\Pi)})-\sum_{p<\infty}\log p\sum_{j=1}^d\sum_{k=1}^\infty p^{k(\overline{\mu_j^p(\Pi)}-s)}\right)ds.
\end{multline*}
We use that the parameters of $\tilde{\Pi}$ are complex conjugate of parameters of $\Pi$.

Recall the average conductor $C_{\rho,X}$ and the test function $\psi$ from the statement of Theorem \ref{low-lying-zero}. We choose
$$\psi_0(s):=\psi\left((s-1/2)\frac{\log C_{\rho,X}}{2\pi i}\right).$$
This choice is admissible as $\psi$ is (the analytic continuation of) a Fourier transform of a smooth function supported on $[-\delta,\delta]$. Then 
$$\int_{(1/2)}\psi_0(s)x^{-s}ds=\frac{x^{-1/2}}{\log C_{\rho,X}}\hat{\psi}\left(\frac{\log x}{\log C_{\rho,X}}\right),\quad\int_{(1/2)}\psi_0(1-s)x^{-s}ds=\frac{x^{-1/2}}{\log C_{\rho,X}}\hat{\psi}\left(\frac{-\log x}{\log C_{\rho,X}}\right),$$
where $\hat{\psi}$ is the Fourier transform of $\psi$. Now
Langlands functoriality in Conjecture \ref{conj-langlands} implies that $\mu_j^\infty(\Pi)$ are away from the poles of $\frac{\Gamma'_\R}{\Gamma_\R}(s+\mu_j^\infty(\Pi))$.
We use Stirling's estimate for $s\in i\R$ to obtain
$$\frac{\Gamma_\R'}{\Gamma_\R}(1/2+s+\mu)=\frac{1}{2}\log|1/2+s+\mu|+O(1)=\log(1+|\mu|)+O(\log(1+|s|)),$$
uniformly in $\mu$. Thus
\begin{align*}\label{estimate-archimedean-part}
    \int_{(1/2)}\psi_0(s)\sum_{j=1}^d\frac{\Gamma'_\R}{\Gamma_\R}(s+\mu_j^\infty(\Pi))ds
    &=\frac{1}{2}\tilde{\psi}_0(1)\log C(\Pi)+O\left(\int_{\R}|{\psi}_0(1/2+it)|\log(1+|t|)dt\right)\nonumber\\
    &=\hat{\psi}(0)\frac{\log C(\Pi)}{2\log C_{\rho,X}}+O((\log X)^{-2}).
\end{align*}
Doing a similar computation we get
$$\int_{(1/2)}\psi_0(1-s)\sum_{j=1}^d\frac{\Gamma'_\R}{\Gamma_\R}(s+\overline{\mu_j^\infty(\Pi)})ds
=\hat{\psi}(0)\frac{\log C(\Pi)}{2\log C_{\rho,X}}+O((\log X)^{-2}).$$
We also define the moment of the Satake parameters by
\begin{equation*}
    \beta_k^p(\Pi):=\sum_{j=1}^dp^{k\mu_j^p(\Pi)},\quad \beta_k(\tilde{\Pi}):=\sum_{j=1}^dp^{k\overline{\mu_j^p(\Pi)}}=\overline{\beta_k^p(\Pi)}.
\end{equation*}
As before, we write zeros of $\Lambda(.,\Pi)$ as $1/2+i\gamma_\Pi$ and poles as $1/2+i\tau_\Pi$. We rewrite \eqref{cauchy-argument-principle} to obtain the \emph{explicit formula}
\begin{multline}\label{explicit-formula}
    \sum_{\gamma_\Pi}\psi\left(\gamma_\Pi\frac{\log C_{\rho,X}}{2\pi}\right)-\sum_{\tau_\Pi}\psi\left(\tau_\Pi\frac{\log C_{\rho,X}}{2\pi i}\right)=\hat{\psi}(0)\frac{\log C(\Pi)}{\log C_{\rho,X}}+O((\log X)^{-2})\\-\sum_{k=1}^\infty\sum_{p<\infty}\frac{p^{-k/2}\log p}{\log C_{\rho,X}}\left({\beta^p_k(\Pi)\hat{\psi}\left(\frac{k\log p}{\log C_{\rho,X}}\right)+\overline{\beta^p_k(\Pi)}}\hat{\psi}\left(\frac{-k\log p}{\log C_{\rho,X}}\right)\right),
\end{multline}

Note that, Langlands functoriality in Conjecture \ref{conj-langlands} implies that 
$$|p^{\mu_j^p(\Pi)}|=1.$$
to truncate the last $k$-sum at $k\le 2$ with an error of $O((\log X)^{-1})$. Finally, summing over the spectrum $\pi\in \hat{\X}_\gen$ in \eqref{explicit-formula} with the weight $\frac{J_X(\pi)}{\ell(\pi)}$ we obtain an expression for \emph{$1$-level zero density} statistic
\begin{multline}\label{statistic-expression}
    \left(D_{\rho,X}(\psi)-D^{\mathrm{pole}}_{\rho,X}(\psi)\right)\int_{\hat{\X}}\frac{J_X(\pi)}{\ell(\pi)}d\mu_\aut(\pi)=\hat{\psi}(0)\int_{\hat{\X}}\frac{J_X(\pi)}{\ell(\pi)}d\mu_\aut(\pi)\\-\frac{1}{\log C_{\rho,X}}\sum_{k=1}^{2}\sum_{p<\infty}\frac{\log p}{p^{k/2}}\hat{\psi}\left(\frac{k\log p}{\log C_{\rho,X}}\right)\int_{\hat{\X}_\gen}\beta^p_k(\rho_*\pi)\frac{J_X(\pi)}{\ell(\pi)}d\mu_\aut(\pi)\\-\frac{1}{\log C_{\rho,X}}\sum_{k=1}^{2}\sum_{p<\infty}\frac{\log p}{p^{k/2}}\hat{\psi}\left(\frac{-k\log p}{\log C_{\rho,X}}\right)\int_{\hat{\X}_\gen}\overline{\beta^p_k(\rho_*{\pi})}\frac{J_X(\pi)}{\ell(\pi)}d\mu_\aut(\pi)+O\left(\frac{X^{r-1}}{\log X}\right).
\end{multline}
The error bound follows from Theorem \ref{orthogonality-full-spectrum}.

\subsection{Moments of Satake parameters}
We use standard representation theory of $\SU(r)$ to understand the quantities $\beta^p_k(\rho_*\pi)$ and their averages over the representations $\pi$.
\begin{lemma}\label{moments-satake-parameter}
Let $\rho, \theta, \mathfrak{s}$ be as in Theorem \ref{low-lying-zero} and $k\in \N$.
There exist coefficients $c_\alpha$ with $\alpha:=(\alpha_1,\dots,\alpha_{r-1})\in \Z^{r-1}_{\ge 0}$ such that
$$\beta^p_k(\rho_*\pi)=c_0(k,\theta)+\sum_{0<\sum\alpha_i\le k(\theta_1-\theta_r)}c_\alpha(k,\theta)\lambda_\pi(p^\alpha),$$
such that $c_0(1,\theta)=0$ and $c_0(2,\theta)=\mathfrak{s}(\rho)$.
\end{lemma}

\begin{proof}
We have,
$$\beta_k^p(\rho_*\pi)=\mathrm{Tr}(\rho^k(\nu))=\mathrm{Tr}(\rho(\nu^k))=\chi_\theta(\nu^k),$$
where $\nu$ is the diagonal matrix in $\SL_r(\C)$ with entries $\nu_i^p(\pi)$ such that $\{\nu^p_i(\pi)\}_{i=1}^r$ are the Satake parameters of $\pi$ at $p$, and $\chi_\theta$ is the character of the highest weight representation of $\SU(r)$ attached to $\theta$. Any symmetric polynomial in $\{\nu_j\}_j$ can be written as a finite linear combination of characters attached to the highest weight representations (the Schur polynomials), evaluated at $\{\nu_j\}_j$. Thus
$$\chi_\theta(\nu^k)=\sum_{\gamma}c_\gamma(k,\theta)\chi_\gamma(\nu),$$
where $\gamma$ runs over dominant elements in $\Z^r/\Z(1,\dots,1)$. We write every weight $\gamma:=(\gamma_1,\dots,\gamma_r)$ appearing in the above equation in $(\gamma_1-\gamma_r,\dots,\gamma_{r-1}-\gamma_r,0)\in \Z^r_{\ge 0}$. The lexicographically highest weight of the polynomial expression in the LHS is clearly $k(\theta_1-\theta_r,\dots,\theta_{r-1}-\theta_r,0)$. By comparing the highest weights in the polynomial expansion of the both sides we can rewrite the last display as
$$\chi_\theta(\nu^k)=\sum_{\gamma:\gamma_1-\gamma_r\le k(\theta_1-\theta_r)}C_\gamma(k,\theta)\chi_\gamma(\nu).$$
We calculate the summand corresponding to the zero weight using orthogonality of characters of $\SU(r)$ with respect to its probability Haar measure:
$$c_0(k,\theta)=\int_{\SU(r)}\chi_\theta(g^k)dg.$$
Thus, clearly $c_0(1,\theta)=0$ and $c_0(2,\theta)=\mathfrak{s}(\rho)$. We conclude the proof writing $\alpha=\iota(\gamma)$ and the relation between $\lambda$ and $\chi$ as in Lemma \ref{casselman-shalika}.
\end{proof}

\begin{lemma}\label{prime-number-theorem}
Let $\psi,\rho,\theta$ be as in Theorem \ref{low-lying-zero} and $c_0(k,\theta)$ as in Lemma \ref{moments-satake-parameter}. Then
$$\frac{1}{\log C_{\rho,X}}\sum_{k=1}^{2}c_0(k,\theta)\sum_{p<\infty}\frac{\log p}{p^{k/2}}\left(\hat{\psi}\left(\frac{k\log p}{\log C_{\rho,X}}\right)+\hat{\psi}\left(\frac{-k\log p}{\log C_{\rho,X}}\right)\right)=\frac{\mathfrak{s}(\rho)}{2}\psi(0)+O\left(\frac{1}{\log X}\right),$$
as $X$ tends off to infinity.
\end{lemma}

\begin{proof}
The summand for $k=1$ vanishes by Lemma \ref{moments-satake-parameter}.
Employing the value of $c_0(2,\theta)$ from Lemma \ref{moments-satake-parameter} thus it will be enough to show that
$$\frac{1}{\log C_{\rho,X}}\sum_{p<\infty}\frac{\log p}{p}\hat{\psi}\left(\frac{2\log p}{\log C_{\rho,X}}\right)=\frac{1}{2}\int_0^\infty \hat{\psi}(t)dt+O\left(\frac{1}{\log X}\right),$$
as changing $\log p$ to $-\log p$ in the entry of $\hat{\psi}$ the main term in the RHS above would be $\int_{-\infty}^0\hat{\psi}(t)dt$. Adding these two main terms we obtain $$\int_{-\infty}^\infty\hat{\psi}(t)dt =\psi(0),$$ as required.

Let $\Psi(x):=\sum_{n\le x}\Lambda(n)$ be the Chebyshev's function where $\Lambda$ is the Von Mangoldt function (not to be confused with the completed $L$-function $\Lambda$). A strong form of the prime number theorem implies that
$$\Psi(x)=x(1+O((1+\log x)^{-2})).$$
We write the sum as a Stieltjes integral,
$$\frac{1}{\log C_{\rho,X}}\sum_{p<\infty}\frac{\log p}{p}\hat{\psi}\left(\frac{2\log p}{\log C_{\rho,X}}\right)=\frac{1}{\log C_{\rho,X}}\int_1^\infty \hat{\psi}\left(\frac{2\log x}{\log C_{\rho,X}}\right)\frac{d\Psi(x)}{x}.$$
Doing an integration by parts and a change of variable the RHS of the above can be written as
$$\frac{1}{2}\int_0^\infty \hat{\psi}(t)(1+(1+t)^{-2})dt-\frac{1}{\log C_{\rho, X}}\int_0^\infty\hat{\psi}'(t)(1+O(1+t)^{-2})dt.$$
This can be easily checked to be equal to 
$$\frac{1}{2}\int_0^\infty \hat{\psi}(t)dt+O(1/\log X),$$
hence the claim follows.
\end{proof}

\begin{proof}[Proof of Theorem \ref{low-lying-zero}]
We calculate the RHS of \eqref{statistic-expression}.
Using Lemma \ref{moments-satake-parameter} we write
\begin{multline*}
    \int_{\hat{\X}}\beta^p_k(\pi)\frac{J_X(\pi)}{\ell(\pi)}d\mu_\aut(\pi)=c_0(k,\theta)\int_{\hat{\X}}\frac{J_X(\pi)}{\ell(\pi)}d\mu_\aut(\pi)\\+\sum_{0<\sum\alpha_i\le k(\theta_1-\theta_r)}c_\alpha(k,\theta)\int_{\hat{\X}}\lambda_\pi(p^\alpha)\frac{J_X(\pi)}{\ell(\pi)}d\mu_\aut(\pi).
\end{multline*}
If $p^{k(\theta_1-\theta_r)}=o(X)$ then using Theorem \ref{orthogonality-full-spectrum} we conclude the second summand in the RHS above vanishes.
However, $\hat{\psi}$ is supported on $[-\delta,\delta]$. Conjecture \ref{conj-cond-size} and the size of $\delta$ as in the statement of Theorem \ref{low-lying-zero} imply that
$$p^{k(\theta_1-\theta_r)}\le C_{\rho,X}^{\delta(\theta_1-\theta_r)}\ll_\rho X^{\delta\mathfrak{C}(\rho)(\theta_1-\theta_r)}=o(X).$$
Hence the second summand in the RHS of \eqref{statistic-expression} is
\begin{multline*}
    \sum_{k=1}^{2}\sum_{p<\infty}\frac{\log p}{p^{k/2}}\hat{\psi}\left(\frac{k\log p}{\log C_{\rho,X}}\right)\int_{\hat{\X}_\gen}\beta^p_k(\rho_*\pi)\frac{J_X(\pi)}{\ell(\pi)}d\mu_\aut(\pi)\\=\sum_{k=1}^{2}c_0(k,\theta)\sum_{p<\infty}\frac{\log p}{p^{k/2}}\hat{\psi}\left(\frac{k\log p}{\log C_{\rho,X}}\right)\int_{\hat{\X}_\gen}\frac{J_X(\pi)}{\ell(\pi)}d\mu_\aut(\pi).
\end{multline*}
Doing a similar analysis with with the third summand in the RHS of \eqref{statistic-expression} we conclude from \eqref{statistic-expression} that $D_{\rho,X}(\psi)-D^{\mathrm{pole}}_{\rho,X}(\psi)$ equals to
\begin{multline*}
\hat{\psi}(0)-\frac{1}{\log C_{\rho,X}}\sum_{k=1}^{2}c_0(k,\theta)\sum_{p<\infty}\frac{\log p}{p^{k/2}}\left(\hat{\psi}\left(\frac{k\log p}{\log C_{\rho,X}}\right)+\hat{\psi}\left(\frac{-k\log p}{\log C_{\rho,X}}\right)\right)+O\left(\frac{1}{\log X}\right).
\end{multline*}
We conclude the proof employing Lemma \ref{prime-number-theorem}.
\end{proof}

\begin{remark}\label{rmk-d-pole}
We expand the remark after Theorem \ref{low-lying-zero} that $D^{\mathrm{pole}}$ is typically negligible. It follows from the following conjecture, as in \cite[Hypothesis 11.2]{ST} (also see the comment afterwards), that $\Lambda(.,\rho_*\pi)$ is entire for almost all $\pi$. Quantitatively, we conjecture the following weighted version.

\begin{conj}\label{conj-pole}
Let $\mathcal{P}$ denote the characteristic function on $\hat{\X}_\gen$ such that $\Lambda(s,\rho_*\pi)$ has a pole in the critical strip. Then
$$\int_{\hat{\X}_\gen}\mathcal{P}(\pi)\frac{J_X(\pi)}{\ell(\pi)}d\mu_\aut(\pi)\ll X^{r-1-\eta(\rho)},$$
for some $\eta(\rho)>0$.
\end{conj}
Note that, because of the Langlands functoriality in Conjecture \ref{conj-langlands} it makes sense to talk about the poles of $\Lambda(.,\rho_*\pi)$. Conjecture \ref{conj-pole} is most likely true. In fact, if $\rho$ is the standard representation, we know that $\Lambda(s,\pi)$ will have a pole only if $\pi$ is non-cuspidal. A standard Weyl law for non-cuspidal spectrum of $G$ supports Conjecture \ref{conj-pole}.

\begin{lemma}\label{poles-negligible}
We assume Conjecture \ref{conj-langlands}, Conjecture \ref{conj-cond-size}, and Conjecture \ref{conj-pole}. Let $\delta<\frac{4\pi\eta(\rho)}{\mathfrak{C}(\rho)}$ where $\mathfrak{C}(\rho)$ and $\eta(\rho)$ are as in Conjecture \ref{conj-cond-size} and Conjecture \ref{conj-pole}, respectively. Then
$$\D^{\mathrm{pole}}_{\rho,X}(\psi)\ll_N (\log X)^{-N}.$$
for all large $N>0$.
\end{lemma}

\begin{proof}
$\Lambda(s,\rho_*\pi)$ has bounded number of poles in the critical strip. We write poles as $1/2+i\tau$ so that $|\Im(\tau)|<1/2$. We recall that $\hat{\psi}$ is supported in $[-\delta,\delta]$. Doing integration by parts sufficiently many times we get
\begin{align*}
&\sum_{\tau}\psi\left(\tau\frac{\log C_{\rho,X}}{2\pi i}\right)\ll \max_{\tau}\left|\psi\left(\tau\frac{\log C_{\rho,X}}{2\pi}\right)\right|\\
&\ll_{N,\psi} C_{\rho,X}^{\delta\frac{\Im(\tau)}{2\pi}}(1+|\tau|\log C_{\rho,X})^{-N}\ll \frac{X^{\delta\mathfrak{C}(\rho)/4\pi}}{(1+\log X)^N}.
\end{align*}
The last estimate follows from Conjecture \ref{conj-cond-size}. Finally, conjecture \ref{conj-pole} immediately yield the claim if $\delta\mathfrak{C}(\rho)<4\pi\eta(\rho)$.
\end{proof}
Thus, if we further assume Conjecture \ref{conj-pole} then the assertion of Theorem \ref{low-lying-zero} can be rewritten as
$$D_{\rho,X}(\psi)=\hat{\psi}(0)-\frac{\mathfrak{s}(\rho)}{2}\psi(0)+O\left(\frac{1}{\log X}\right),$$
if $\delta<\frac{1}{\mathfrak{C}(\rho)}\min\left(\frac{1}{\theta_1-\theta_r},4\pi\eta(\rho)\right)$.
\end{remark}

\section{The Orthogonality Conjecture over a Cuspidal Spectrum}
We give a variant of Theorem \ref{orthogonality-full-spectrum} (that is, the conjecture in \eqref{conjecture-orthogonality}) over a cuspidal spectrum as discussed in Remark \ref{project-on-cusp}. We also mention that a similar technique may be applied to prove a variant of Theorem \ref{sato-tate} over a cuspidal spectrum.

In this section we will work in the $S$-arithmetic setting. We fix $p$ to be a fixed finite prime. We redefine our notations. Let $G:=\PGL_r(\R)\times\PGL_r(\Q_p)$, $\Gamma:=\PGL_r(\Z[1/p])$ diagonally embedded in $G$, and $\X:=\Gamma\backslash G$. The strong approximation theorem on $G$ implies that
$$\X\cong \PGL_r(\Q)\backslash\PGL_r(\mathbb{A})/\prod_{v\neq p}\PGL_r(\Z_v),$$
where $\mathbb{A}$ is the ring of adeles over $\Q$. For any irreducible automorphic representation $\pi$ of $G$ let $\pi_p$ and $\pi_\infty$ be $p$-adic and infinite components of $\pi$, respectively. 

Let $N$ be the unipotent subgroup in $G$ and $\psi_m:=\psi_{m,\infty}\otimes\psi_{m,p}$ for $m\in \Z[1/p]^{r-1 }_{\neq 0}$ is an additive character on $N$, where $\psi_{m,\infty}$ is defined as in \eqref{additive-char-unip} and $\psi_{.,p}$ is induced similarly from an unramified character $\psi_0$ of $\Q_p$. We abbreviate $\psi_{v}=\psi_{m,v}$ for $m=(1,\dots,1)$, $v=p,\infty$.

We can realize a Whittaker model $\W(\pi_\infty,\psi_\infty)$ of $\pi_\infty$. We define $J_F(\pi_\infty)$ (similarly, $J_{F_X}(\pi_\infty)= J_X(\pi_\infty)$ for $F=F_X$ as defined after \eqref{defn-congruence-subgroup}) similarly as in \eqref{bessel-distribution}. $J_X$, as usual, enjoys similar property as in Lemma \ref{spectral-weight-property}. We also define $W_F$ (similarly, $W_{F_X}=W_X$ for $F=F_X$) as in \eqref{psi-average} with respect to the additive character $\psi_\infty$ of $N(\R)$.
We provide a similar orthogonality statement as in Theorem \ref{orthogonality-full-spectrum} varying only over the cusp forms, as follows.

\begin{theorem}\label{orthogonality-cuspidal-spectrum}
Let $X$ be large. Let $\sigma$ be a fixed supercuspidal representation of $\PGL_r(\Q_p)$. Then for $m,n\in\N^{r-1}$ coprime to $p$ (i.e. all the coordinates are coprime to $p$) with
$$\min(n_{r-1}^{r-1}\dots n_1m_1^{r-1}\dots m_{r-1},m_{r-1}^{r-1}\dots m_1 n_1^{r-1}\dots n_{r-1})\ll X^r$$
with sufficiently small implied constant, we have
$$\sum_{\substack{{\pi\,\,\mathrm{cuspidal}}\\{\pi_p=\sigma}}}{\overline{\lambda_\pi(m)}\lambda_\pi(n)}\frac{J_X(\pi_\infty)}{L^{(p)}(1,\pi,\Ad)}=\epsilon_p(1,\sigma\otimes\tilde{\sigma}) W_X(1)\delta_{m=n},$$
where $\epsilon_p$ is the $p$-adic $\epsilon$-factor, $W_X(1)\asymp X^{r-1}$ 
and $L^{(p)}$ denotes the partial $L$-function excluding the Euler factor at $p$.
\end{theorem}

\subsection{Proof of Theorem \ref{orthogonality-cuspidal-spectrum}}
As before, for a subgroup $H<G$ we define $\Gamma_H:=\Gamma\cap H$. For a generic irreducible unitary automorphic representation $\pi\in \hat{\X}_\gen$ and a factorizable $\varphi\in \pi$ we may attach a Whittaker function $W_\varphi:=W_{\varphi,p}W_{\varphi,\infty}$ to $\varphi$.
\begin{equation}
    \int_{[N]}\varphi(x(g_\infty,g_p))\overline{\psi_m(x)}dx=\frac{\lambda_\pi(m_o)}{\delta^{1/2}(\tilde{m})\delta^{1/2}(\tilde{m}_p)}W_{\varphi,p}(\tilde{m}g_p)W_{\varphi,\infty}(\tilde{m}g_\infty),
\end{equation}
where $m_p$ and $m_o$ are the $p$-adic part and the $p$-coprime part of $\tilde{m}$, respectively\footnote{That is, if $m=(m_1,\dots,m_{r-1})$ and $m_i= p^{k_i}m'_i$ then $m_o:=(m'_1,\dots,m'_{r-1})$ and $m_p:=(p^{k_1},\dots, p^{k_r-1})$.}. A similar formula as in \eqref{defn-ell-pi} holds in this case as well, where in this case \cite[\S2.2]{LO}
\begin{equation}
    \|W_{\varphi}\|=\|W_{\varphi,p}\|_{\W(\pi_p,\psi_p)}\|W_{\varphi,\infty}\|_{\W(\pi_\infty,\psi_\infty)},
\end{equation}
and $\ell(\pi)\asymp L^{(p)}(1,\pi,\Ad)$ with an absolute implied constant, if $\pi$ is cuspidal.

Let $\sigma$ be a supercuspidal representation of $\PGL_r(\Q_p)$ with conductor exponent of $\sigma$ being fixed. A particular choice of $\sigma$ will be fixed in the next section. Let $\Phi_\sigma$ be the matrix coefficient of $\sigma$ as defined in \eqref{defn-matrix-coeff}. Also recall $F_X$, a smoothened $L^1$-normalized characteristic function of $K_0(X,\tau)$ as in \eqref{defn-congruence-subgroup} and the discussion afterwards. Then a very similar argument as in the proof of Proposition \ref{kuznetsov} would yield
\begin{multline}\label{kuznetsov-s-adic}
    \int_{\hat{\X}_\gen}\frac{\overline{\lambda_\pi(m_o)}\lambda_\pi(n_o)}{\delta^{1/2}(\tilde{m}_p)\delta^{1/2}(\tilde{n}_p)}\frac{J_X(\pi_\infty)J_\sigma(\pi_p)}{\ell(\pi)}d\mu_\aut(\pi)\\=\delta_{m=n}W_X(1)W_\sigma(1)+\sum_{1\neq w\in W}\frac{\delta^{1/2}(\tilde{n})}{\delta^{1/2}(\tilde{m})\delta_w(\tilde{n})}\sum_{c\in \Z[1/p]^{r-1}_{\neq 0}}S_w(m,n;c)\\
    \times\int_{N_w(\R)}W_X(\tilde{m}c^*w\tilde{n}^{-1}x)\overline{\psi_{\infty}(x)}dx\int_{N_w(\Q_p)}W_\sigma(\tilde{m}c^*w\tilde{n}^{-1}x)\overline{\psi_{p}(x)}dx,
\end{multline}
where, similar to \eqref{bessel-distribution}
\begin{equation}\label{bessel-distribution-s-adic}
    J_\sigma(\pi_p)=\sum_{W\in \B(\pi_p)}\pi_p(\bar{\Phi}_\sigma)W(1)\overline{W(1)},
\end{equation}
for some orthonormal basis of $\B(\pi_p)$ of $\pi_p$. Similarly, as in \eqref{psi-average} we define\footnote{We hope that readers don't confuse between $W_\sigma$ and Whittaker functions $W$}
\begin{equation}\label{psi-average-s-adic}
    W_\sigma(g)=\int_N \Phi_\sigma(xg)\overline{\psi_{p}(x)}dx.
\end{equation}

\subsection{Local p-adic computation}
Let $\sigma$ be a supercuspidal representation of $\PGL_r(\Q_p)$ such that the local $L$-factor
\begin{equation}\label{choice-supercuspidal}
    L(s,\sigma\otimes\tilde{\sigma})=(1-p^{-rs})^{-1}.
\end{equation}
Such a supercuspidal representation exists, e.g. see \cite[\S2.1]{Ye}.

Let the matrix coefficient $\Phi_\sigma$ of the supercuspidal representation $\sigma$ is defined by
\begin{equation}\label{defn-matrix-coeff}
    \Phi_\sigma(g)=\langle \sigma(g) W_0,W_0\rangle_{\W(\sigma,\psi_p)},
\end{equation}
where $W_0$ is an $L^2$-normalized vector in the Whittaker model $\W(\sigma,\psi_p)$ of $\sigma$ with respect to $\psi_p$, with $W_0(1)\neq 0$.\footnote{Such a vector exists, e.g. a newvector of $\sigma$.}

\begin{prop}\label{local-s-adic-calculation}
Recall the definitions in \eqref{bessel-distribution-s-adic} and \eqref{psi-average-s-adic}. Then
$$J_\sigma(\pi_p)=\delta_{\pi_p\cong \sigma}\epsilon(1,\sigma\otimes\tilde{\sigma})W_\sigma(1),$$
where $\epsilon$ denotes the local epsilon factor and $\tilde{\sigma}$ is the contragredient of $\sigma$.
\end{prop}

We first prove following lemma about a Fourier transform of the matrix coefficient to prepare the proof of Proposition \ref{local-s-adic-calculation}. The following lemma appeared for $\GL_2(\Q_p)$ in \cite[Lemma 3.5]{Q}, also in \cite[Lemma 3.4.2]{MV}.

\begin{lemma}\label{fourier-transform-matrix-coeff}
Let $\Phi_\sigma$ be the matrix coefficient defined above. Then
$$\int_{N(\Q_p)}\Phi_\sigma(ng)\overline{\psi_p(n)}dn=W_0(g)\overline{W_0(1)}.$$
\end{lemma}

\begin{proof}
As $\sigma$ is supercuspidal $\Phi_\sigma$ is compactly supported in $\PGL_r(\Q_p)$. Thus the integral is absolutely convergent.
Clearly it is enough to prove that
$$\int_{N(\Q_p)}\int_{N_{r-1}(\Q_p)\backslash\GL_{r-1}(\Q_p)}W_1\left[\begin{pmatrix}h&\\&1\end{pmatrix}n\right]\overline{W_2\left[\begin{pmatrix}h&\\&1\end{pmatrix}\right]}\overline{\psi_p(n)}dhdn= W_1(1)\overline{W_2(1)},$$
for $W_1,W_2\in \W(\sigma,\psi_p)$.

We write $n=\begin{pmatrix} I_{r-1}&x\\&1\end{pmatrix}\begin{pmatrix}u&\\&1\end{pmatrix}$ for $u\in N_{r-1}(\Q_p)$ and $x\in \Q_p^{r-1}$. Correspondingly, we write $dn=dudx$, and $\psi_p(n)=\psi_p(u)\psi_0(e_{r-1}x)$, where $e_{r-1}$ is the row vector $(0,\dots,0,1)$ and $\psi_p(u)$ (denoted with an abuse of notation) is the restriction of $\psi_p$ on $N_{r-1}(\Q_p)$. We use unipotent equivariance of $W_1$ to re-write the last integral as
\begin{multline*}
    \int_{N_{r-1}(\Q_p)\backslash\GL_{r-1}(\Q_p)}\overline{W_2\left[\begin{pmatrix}h&\\&1\end{pmatrix}\right]}\int_{N_{r-1}(\Q_p)}W_1\left[\begin{pmatrix}hu&\\&1\end{pmatrix}\right]\overline{\psi_p(u)}du\\
    \times\int_{\Q_p^{r-1}}\psi_0(e_{r-1}hx)\overline{\psi_0(e_{r-1}x)}dxdh.
\end{multline*}
By Fourier inversion the inner-most integral evaluates to $\delta_{e_{r-1}h=e_{r-1}}$ (as a distribution) and thus we can re-write the above as
$$\int_{N_{r-2}(\Q_p)\backslash\GL_{r-2}(\Q_p)}\overline{W_2\left[\begin{pmatrix}h'&\\&I_2\end{pmatrix}\right]}\int_{N_{r-1}(\Q_p)}W_1\left[\begin{pmatrix}\begin{pmatrix}h'&\\&1\end{pmatrix}u&\\&1\end{pmatrix}\right]\overline{\psi_p(u)}dudh'.$$
Proceeding similarly with the $u$-integral and inducting on $r$ we conclude the proof.
\end{proof}

We record that $W_0$ is compactly supported in $\PGL_r(\Q_p)\mod N(\Q_p)$ \cite[Corollary 6.5]{CS}. This fact will be used below.

\begin{lemma}\label{value-of-zeta-integral}
Let $W_0\in \W(\sigma,\psi_p)$ with $\|W_0\|_{\W(\sigma,\psi_p)}=1$. Then
$$\int_{N(\Q_p)\backslash\PGL_r(\Q_p)}|W_0(g)|^2dg=\epsilon(1,\sigma\otimes\tilde{\sigma}),$$
where $\epsilon$ is the local $\epsilon$-factor.
\end{lemma}

\begin{proof}
Let $P$ be the maximal parabolic, attached to the partition $r=(r-1)+1$ in $\PGL(r)$. Let $\Psi$ be the characteristic function of $\Z_p^{r}$. We define a function on $\PGL_r(\Q_p)$ by
$$f_{s,\Psi}(g):=\int_{\Q_p^\times}\Psi(te_rg)|\det(tg)|^{s}d^\times t,$$
which is defined for $\Re(s)>0$ and then can be meromorphically continued.
Clearly, $f_{s,\Psi}$ is spherical (i.e. $\PGL_r(\Z_p)$ invariant) and
$$f_{s,\Psi}\left[\begin{pmatrix}h&\\&1\end{pmatrix}\right]=|\det(h)|^sf_{s,\Psi}(1),\quad h\in \GL_{r-1}(\Q_p).$$
We use the $\PGL(r)\times\PGL(r)$ local functional equation as in \cite[2.5 Theorem]{JPSS2} to obtain
\begin{equation}\label{local-functional-equation}
    \int_{N(\Q_p)\backslash\PGL_r(\Q_p)}|W_0(g)|^2f_{s,\Psi}(g)dg=\gamma(1-s,\sigma\otimes\tilde{\sigma})\int_{N(\Q_p)\backslash\PGL_r(\Q_p)}|\tilde{W}_0(g)|^2f_{1-s,\hat{\Psi}}(g)dg,
\end{equation}
where $\tilde{W}_0$ is the contragredient of $W_0$ and $\hat{\Psi}$ is the Fourier transform of $\Psi$, which also equals to $\Psi$. Here the local $\gamma$-factor using \eqref{choice-supercuspidal} can be written as
\begin{equation}\label{explicit-gamma-factor}
    \gamma(1-s,\sigma\otimes\tilde{\sigma})=\epsilon(1-s,\sigma\otimes\tilde{\sigma})\frac{(1-p^{-rs})^{-1}}{(1-p^{rs-r})^{-1}},
\end{equation}
where $\epsilon$ denote the local $\epsilon$-factor which is entire as a function of $s$.

We first look at the integral in the RHS of \eqref{local-functional-equation}. Using Iwasawa decomposition we write $N(\Q_p)\backslash\PGL_r(\Q_p)\ni g=\begin{pmatrix}a&\\&1\end{pmatrix}k$ where $a\in A_{r-1}(\Q_p)$ the subgroup of the diagonal matrices in $\GL_{r-1}(\Q_p)$ and $k\in \PGL_r(\Z_p)$. Correspondingly, we write 
$$dg=\frac{d^\times a}{\delta\left[\begin{pmatrix}a&\\&1\end{pmatrix}\right]}dk,\quad d^\times a:=\prod d^\times a_i, a:=\mathrm{diag}(a_1,\dots,a_{r-1}).$$
Note that $\delta\left[\begin{pmatrix}a&\\&1\end{pmatrix}\right]=\delta(a)|\det(a)|$.
Thus we can write the RHS of \eqref{local-functional-equation} as
$$f_{1-s,\hat{\Psi}}(1)\int_{A_{r-1}(\Q_p)\times\PGL_r(\Z_p)}\left|\tilde{W}_0\left[\begin{pmatrix}a&\\&1\end{pmatrix}k\right]\right|^2|\det(a)|^{1-s}\frac{d^\times a}{\delta(a)|\det(a)|}dk.$$
We change variable $k\mapsto\begin{pmatrix}k'&\\&1\end{pmatrix}k$ with $k'\in \GL_{r-1}(\Z_p)$ and average over $\GL_{r-1}(\Z_p)$ (we normalize the measures so that $\mathrm{vol}(\GL_s(\Z_p))=1$ for $1\le s\le r$) to obtain the above equals to
$$f_{1-s,\hat{\Psi}}(1)\int_{\PGL_r(\Z_p)}\int_{N_{r-1}(\Q_p)\backslash\GL_{r-1}(\Q_p)}\left|\tilde{W}_0\left[\begin{pmatrix}h&\\&1\end{pmatrix}k\right]\right|^2|\det(h)|^{-s}dh dk.$$
We record that for $s=0$, using the invariance of the unitary product, the above double integral equals to $\|\tilde{W}_0\|^2_{\W(\tilde{\sigma},\overline{\psi_p})}$.

On the other hand, doing a similar computation as above we can obtain that the LHS of \eqref{local-functional-equation} is 
$$f_{s,\Psi}(1)\int_{A_{r-1}(\Q_p)\times\PGL_r(\Z_p)}\left|W_0\left[\begin{pmatrix}a&\\&1\end{pmatrix}k\right]\right|^2|\det(a)|^{s}\frac{d^\times a}{\delta(a)|\det(a)|}dk.$$
Note that the integral above is absolutely convergent for all $s$ due to the compact support of $W_0$ in the domain of integration. Moreover, the integral tends to 
$$\int_{N(\Q_p)\backslash\PGL_r(\Q_p)}|W_0(g)|^2dg,$$
as $s\to 0$.

We let $s\to 0$ from the right (i.e. $\Re(s)>0$) and use $\|\tilde{W}_0\|^2_{\W(\tilde{\sigma},\overline{\psi_p})}=\|W_0\|_{\W(\sigma,\psi_p)}=1$ to conclude from \eqref{local-functional-equation} that
$$\int_{N(\Q_p)\backslash\PGL_r(\Q_p)}|W_0(g)|^2dg=\lim_{s\to 0}\gamma(1-s,\sigma\otimes\tilde{\sigma})\frac{f_{1-s,\hat{\Psi}}(1)}{f_{s,\psi}(1)}.$$
We compute 
\begin{align*}
    f_{s,\Psi}(1)&=\int_{\Q_p^\times}\Psi(0,\dots,0,t)|t|^{rs}d^\times t=(1-p^{-rs})^{-1},
\end{align*}
and similarly,
$$f_{1-s,\hat{\psi}}(1)=(1-p^{rs-r})^{-1}.$$
Thus from \eqref{explicit-gamma-factor} we see that
$$\gamma(1-s,\sigma\otimes\tilde{\sigma})\frac{f_{1-s,\hat{\Psi}}(1)}{f_{s,\psi}(1)}=\epsilon(1-s,\sigma\otimes\tilde{\sigma}).$$
We conclude the proof by taking $s=0$.
\end{proof}

\begin{proof}[Proof of Proposition \ref{local-s-adic-calculation}]
We first note that Lemma \ref{fourier-transform-matrix-coeff} implies that
$$W_\sigma(1)=|W_0(1)|^2.$$
For $W\in \W(\pi_p,\psi_p)$ we consider the $\GL_r(\Q_p)$-invariant pairing between $\pi_p$ and $\sigma$ by
\begin{equation}\label{invariant-pairing}
    \int_{N(\Q_p)\backslash\PGL_r(\Q_p)}W(g)\overline{W_0(g)}dg,
\end{equation}
where $dg$ is the projection of the Haar measure on $\GL_r(\Q_p)$.
Note that the integral in \eqref{invariant-pairing} is absolutely convergent as $W_0$ is compactly supported in the domain of the integration.

Schur's lemma implies that \eqref{invariant-pairing} is non-zero only if $\pi_p\cong \sigma$ and if $\pi_p\cong \sigma$ then \eqref{invariant-pairing} is proportional to the invariant unitary product in $\sigma$ given by
$$\int_{N_{r-1}(\Q_p)\backslash\GL_{r-1}(\Q_p)}W\left[\begin{pmatrix}h&\\&1\end{pmatrix}\right]\overline{W_0\left[\begin{pmatrix}h&\\&1\end{pmatrix}\right]}dh.$$


Recalling \eqref{bessel-distribution-s-adic} and folding the $\PGL_r$ integral over $N$ we can write
$$J_\sigma(\pi_p)=\sum_{W\in \B(\pi_p)}\overline{W(1)}W_0(1)\int_{N(\Q_p)\backslash\PGL_r(\Q_p)}W(g)\int_{N(\Q_p)}\overline{\Phi_\sigma(ng)}\psi(n)dndg.$$
Using Lemma \ref{fourier-transform-matrix-coeff} we obtain
$$J_\sigma(\pi_p)=\sum_{W\in \B(\pi_p)}\overline{W(1)}W_0(1)\int_{N(\Q_p)\backslash\PGL_r(\Q_p)}W(g)\overline{W_0(g)}dg.$$
Hence $J_\sigma(\pi_p)=0$ if $\pi_p$ is not isomorphic to $\sigma$, and if $\pi_p=\sigma$, choosing an orthonormal basis $\B(\sigma)\ni W_0$ we obtain
$$J_\sigma(\sigma)=|W_0(1)|^2\int_{N(\Q_p)\backslash\PGL_r(\Q_p)}|W_0(g)|^2dg.$$
We conclude the proof using Lemma \ref{value-of-zeta-integral}.
\end{proof}

\begin{proof}[Proof of Theorem \ref{orthogonality-cuspidal-spectrum}]
We apply \eqref{kuznetsov-s-adic} for $m,n$ coprime with $p$. A similar argument as in Theorem \ref{orthogonality-full-spectrum} would imply that all terms in the geometric side of \eqref{kuznetsov-s-adic} corresponding to non-trivial Weyl terms will vanish. We conclude the proof after applying Lemma \ref{local-s-adic-calculation}.
\end{proof}


\end{document}